\documentclass[11pt]{amsart}

\usepackage{etex}
\usepackage{amsmath, amssymb}
\usepackage[frame,cmtip,arrow,matrix,line,graph,curve]{xy}
\usepackage{graphpap, color, paralist, pstricks}
\usepackage[mathscr]{eucal}
\usepackage[pdftex]{graphicx}
\usepackage[pdftex,colorlinks,backref=page,citecolor=blue]{hyperref}
\usepackage{tikz}

\newtheorem{theorem}{Theorem}[section]
\newtheorem{proposition}[theorem]{Proposition}
\newtheorem{corollary}[theorem]{Corollary}
\newtheorem{lemma}[theorem]{Lemma}

\newtheorem{conjecture}[theorem]{Conjecture}
\newtheorem{problem}[theorem]{Problem}

\theoremstyle{definition}
\newtheorem{definition}[theorem]{Definition}

\newtheorem{remark}[theorem]{Remark}

\newcommand{\PP}{\mathbb{P}}
\newcommand{\QQ}{\mathbb{Q}}
\newcommand{\CC}{\mathbb{C}}
\newcommand{\RR}{\mathbb{R}}
\newcommand{\ZZ}{\mathbb{Z}}

\newcommand{\HH}{\mathbb{H}}

\newcommand{\cO}{\mathcal{O} }
\newcommand{\cA}{\mathcal{A} }

\newcommand{\cC}{\mathcal{C} }
\newcommand{\cE}{\mathcal{E} }
\newcommand{\cF}{\mathcal{F} }
\newcommand{\cG}{\mathcal{G} }

\newcommand{\cM}{\mathcal{M} }
\newcommand{\cN}{\mathcal{N} }
\newcommand{\cL}{\mathcal{L} }
\newcommand{\cQ}{\mathcal{Q} }
\newcommand{\cS}{\mathcal{S} }
\newcommand{\cT}{\mathcal{T} }
\newcommand{\cU}{\mathcal{U} }
\newcommand{\cW}{\mathcal{W} }
\newcommand{\cZ}{\mathcal{Z} }

\newcommand{\rA}{\mathrm{A} }
\newcommand{\rH}{\mathrm{H} }
\newcommand{\rM}{\mathrm{M} }
\newcommand{\rN}{\mathrm{N} }
\newcommand{\rQ}{\mathrm{Q} }

\newcommand{\proj}{\mathrm{Proj}\;}

\newcommand{\cHom}{\mathcal{H}om}
\newcommand{\cEnd}{\mathcal{E}nd}

\def\Hom{\mathrm{Hom} }
\def\Ext{\mathrm{Ext} }
\def\cExt{\mathcal{E}xt }

\newcommand{\rank}{\mathrm{rank}\, }

\def\GL{\mathrm{GL}}
\def\git{/\!/ }

\def\lr{\rightarrow}

\def\bH{\mathrm{H}}
\def\bM{\mathrm{M}}
\def\bN{\mathrm{N}}

\newcommand{\ses}[3]{0\lr{#1}\lr{#2}\lr{#3}\lr 0}

\hypersetup{%
pdftitle={Chow ring of $\mathrm{M}(4,1)$},%
pdfauthor={Kiryong Chung, Han-Bom Moon},%
pdfkeywords={moduli of sheaves, minimal model program, moduli space},%
citecolor=blue,%
linkcolor=blue,%
}

\begin{document}

\title{Chow ring of the moduli space of stable sheaves supported on quartic curves}

\author{Kiryong Chung}
\address{School of Mathematics, Korea Institute for Advanced Study, Seoul 130-722, Korea}
\email{krjung@kias.re.kr}

\author{Han-Bom Moon}
\address{Department of Mathematics, Fordham University, Bronx, NY 10458}
\email{hmoon8@fordham.edu}

\keywords{Chow ring; Moduli space of sheaves; Blowing-up/down; Euler charateristic}
\subjclass[2010]{14C15, 14E30, 14N35, 14F05.}

\begin{abstract}
Motivated by the computation of the BPS-invariants on a local Calabi-Yau threefold suggested by S. Katz, we compute the Chow ring and the cohomology ring of the moduli space of stable sheaves of Hilbert polynomial $4m+1$ on the projective plane. As a byproduct, we obtain the total Chern class and Euler characteristics of all line bundles, which provide a numerical data for the strange duality on the plane. 
\end{abstract}

\maketitle


\section{Introduction and results}
\subsection{Introduction}
For a smooth projective variety $X$, let $\rM(X, P(m))$ be the moduli space of (Gieseker-Mumford) semistable sheaves on $X$ with Hilbert polynomial $P(m)$. Since its construction in the monumental paper \cite{Sim94}, the geometry of $\rM(X, P(m))$ has been actively studied by many algebraic geometers in various contexts. In this paper, we will study the case that $X = \PP^{2}$ and $P(m)$ is a linear polynomial. 

The geometry of $\rM(d, \chi) := \rM(\PP^{2}, dm+\chi)$ is studied at least three different directions: birational geometry, curve counting theory, and strange duality. In a group of birational geometers, the log minimal model program of $\rM(d,\chi)$ is actively studied with a connection with the notion of moduli spaces of objects in the derived category of coherent sheaves (\cite{ABCH13,BMW14,CHW14,Liu15,Woo13}). The key ingredients are that one can identify the moduli space of semistable sheaves with the moduli space of Bridgeland stable objects in $\mathrm{D}^b(\PP^2)$ with a proper stability condition, and by varying the stability condition, one can obtain new moduli spaces which are birational to $\rM(d, \chi)$. 

On the other hand, the virtual curve counting theory on Calabi-Yau (CY) threefolds focuses on the computation of certain topological invariants of the space $\rM(d,1)$. By S. Katz (\cite{Kat08}), it was suggested that genus zero BPS-invariants of the total space of the canonical line bundle $K_{\PP^{2}}$ with the canonical polarization are given by Donaldson-Thomas invariants ($=$ the top Chern class of the cotangent bundle) of $\rM(d, 1)$. More generally, there have been several mathematical proposals of the definition of BPS-invariants in \cite{HST01, Kat08, KL12, CKK14} in terms of moduli spaces of sheaves. The study of the space $\rM(d, 1)$ and its cohomology ring provides an explicit and computable example of these theories in a local CY case. 

Finally, in the study of the strange duality, we are interested in the isomorphism between the space of theta divisors in two different moduli spaces of sheaves. The case of moduli spaces of bundles on a curve was established in \cite{Bel08a, MO07}. The theory for surfaces, in particular $\PP^{2}$, was initiated by Le Potier, and there have been some partial progress for instance \cite{Dan00, Dan02, Yua12b}. Generating functions of theta divisors has been studied as generalized K-theoretic Donaldson invariants by G\"ottsche, Nakajima and Yoshioka (\cite{GNY09}). By Grothendieck-Riemann-Roch theorem, studying the cohomology ring and Chern classes of $\rM(d, \chi)$ provides some numerical data for the strange duality. 

Of course, the study of the geometric/topological structure of moduli spaces is itself an interesting problem. In addition, with the motivations from curve counting theory and study of theta divisors, the following geometric question arises naturally.

\begin{problem}
Study the cohomology ring structure of the moduli space $\rM(d,\chi)$.\end{problem}

In \cite{Mar07}, Markman proved that the cohomology ring $\rH^*(\rM(d,\chi))$ is generated by Chern classes of K\"unneth factors of a universal family on $\rM(d,\chi)$, which is similar to the case of moduli spaces of vector bundles on curves. Also he showed that the Chow ring and the cohomology ring are isomorphic. But the lack of general tools dealing with the moduli space $\rM(d,\chi)$ leads us to solve the problem separately for each $d$ and $\chi$. 

In this paper, we study $\rM(d,1)$. In this case the semistability coincides with the stability. If $d=1,2$, then $\rM(d,1) \cong |\cO_{\PP^2}(d)|=\PP^{d(d+3)/2}$ (\cite[Theorem 3.5]{LP93a}). If $d=3$, $\rM(3,1)\cong \cC_3$, the universal cubic plane curve (\cite[Theorem 5.1]{LP93a}). For the cohomology ring ($=$ Chow ring) of $\cC_{3}$, see Section \ref{ssec:univcurve}. Thus the first non-trivial case is $d=4$. 

\subsection{Main result}

The main result of this paper is the following:
\begin{theorem}[\protect{Theorem \ref{thm:ChowM}}]\label{thm:ChowMintro}
Let $\rM := \rM(4,1)$ be the moduli space of stable sheaves with Hilbert polynomial $4m+1$. Then, the Chow ring of $\rM$ is 
\[
\begin{split}
\rA^{*}(\rM) \cong \;& \QQ[\alpha, \beta, x, y, z]/\langle
xz-yz,\beta^2 z-3yz-9z^2,3\alpha^2 z-\alpha\beta z+yz,\beta^2 y-3y^2-9yz,\\
&\beta^2 x-xy-3y^2-3\alpha \beta z-9yz+9z^2, \beta^4+3x^2-9xy-3y^2-54yz-81z^2,\\
&\beta yz+9\alpha z^2-3\beta z^2, 2\beta xy-3\beta y^2-9\alpha yz-27\alpha z^2+9\beta z^2,3\beta x^2-7\beta y^2-36\alpha yz\\
&-108\alpha z^2+36\beta z^2,
\alpha^{12}+3\alpha^{11}\beta+3\alpha^{10}(\beta^2+2x-y)
+\alpha^9(-\beta^3+12\beta x+2\beta y)\\
&+3\alpha^{8}(9x^{2}-16xy+27y^{2})+28\alpha^7\beta y^2
+56\alpha^6 y^3+201\alpha\beta z^5-19yz^5-613z^6,\\
& 6\alpha^{10}xy-12\alpha^{10}y^2-10\alpha^9 \beta y^2-45\alpha^8 y^3-104\alpha\beta z^6+2yz^6+310z^7\rangle
\end{split}
\]
where $\alpha, \beta$ are of degree 1 and $x, y, z$ are of degree 2. 
\end{theorem}

In addition, we describe (some linear combinations of) generators in terms of Chern classes and some natural effective cycles in Section \ref{sec:Cherneff}. A general element $F \in \rM$ has a unique nonzero section $\cO_{\PP^{2}}\stackrel{s}{\to}F$ up to a scalar multiplication, and its cokernel $Q_{F}$ has a finite support. Let $O$ be the closure of the locus of $F \in \rM$ such that $Q_{F}$ contains a fixed point in $\PP^{2}$. Let $L$ be the closure of the locus of $F \in \rM$ such that $Q_{F}$ meets a fixed line. Let $S$ be the locus of $F \in \rM$ such that $\mathrm{supp}(F)$ contains a fixed point. Finally, let $\cC_{4} = \{F \in \rM\;|\; \mathrm{dim}\;\rH^{0}(F) = 2\}$ be the Brill-Noether locus (the notation will be justified in Section \ref{sec:Cherneff}). These natural effective loci are described in terms of above generators explicitly. 
\begin{proposition}[\protect{Propositions \ref{prop:BNlocus}, \ref{prop:alphaclass}, and \ref{prop:relation}}]
\[
	S = \alpha, \; L = -\beta, \;O = x - y, \;\cC_{4} = z.
\]
\end{proposition}

The total Chern class is obtained as a byproduct. For the entire formula, see Proposition \ref{prop:Chernclass}. 
\begin{proposition}[\protect{Proposition \ref{prop:Chernclass}}]
The total Chern class of $\cT_{\rM}$ is described in terms of above generators. For instance, 
\[
	c_{1}(\cT_{\rM}) = 12\alpha, \;
	c_{2}(\cT_{\rM}) = 66\alpha^{2}-3\alpha\beta -3\beta^{2}+
	6x+2y+34z, \cdots.
\]
\end{proposition}

It is straightforward to check that $S = c_{1}(\cT_{\rM})/12, c_{2}(\cT_{\rM}), L, O, \cC_{4}$ generate $\rA^{*}(\rM)$. However, the presentation using this generating set is revealed to be very complicate. Thus we decided to use the above presentation.

The Chern class computation enables us to compute Euler characteristics (generalized K-theoretic Donaldson numbers) of line bundles on $\rM$. We present some of Euler characteristics of determinant line bundles $\lambda(w) \in \mathrm{Pic}(\rM)$ of $w\in [F]^\perp \subset \mathrm{K}(\PP^2)$ for $F\in \rM$. For more details, see Section \ref{sec:Eulerchar}.

\subsection{Idea of proof}

The main idea of the proof of Theorem \ref{thm:ChowMintro} is to make a sequence of moduli spaces which are connected by standard morphisms such as projection and blow-up/down. Then we can apply standard formulas for the computation of Chow ring (\cite{Ful98}). We may summarize relations between moduli spaces appearing in this paper as the following diagram. 
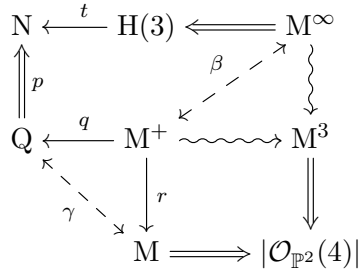
\[
	\xymatrix{\rN&\bH(3)\ar[l]_{t}&\rM^{\infty}\ar@{=>}[l]
	\ar@{~>}[d]\ar@{<-->}[ld]_{\beta}\\
	\rQ \ar@{=>}[u]_{p}&\rM^{+}\ar[l]_{q} \ar[d]^{r}\ar@{~>}[r]
	&\rM^{3}\ar@{=>}[d]\\
	&\rM \ar@{<-->}[lu]^{\gamma} \ar@{=>}[r] 
	&|\cO_{\PP^{2}}(4)|}
\]
The space $\rN = \rN(3;2,3)$ is the moduli space of Kronecker quiver representations (Section \ref{sub:defq}), $\rQ$ is a projective bundle over $\rN$ with 11 dimensional fiber, and $\rH(3)$ is the Hilbert scheme of 3 points on $\PP^{2}$. The space $\rM^{\alpha} := \rM^{\alpha}(\PP^{2}, 4m+1)$ is the moduli space of $\alpha$-stable pairs with Hilbert polynomial $4m+1$ in $\PP^{2}$ (Section \ref{ssec:moduliofpairs}). An ordinary arrow $\to$ refers a smooth blow-up, a double arrow $\Rightarrow$ is a fibration with positive dimensional fibers, and a wiggled arrow $\leadsto$ refers a small contraction, and finally, a dashed arrow $\dashrightarrow$ is a flip. 

Many of these morphisms can be explained in a broader theoretical context. For instance, $\rM^{\infty}(\PP^{2}, dm+\chi)$ is known to be isomorphic to the relative Hilbert scheme of $\left(\chi-\frac{d(3-d)}{2}\right)$-points on the universal degree $d$ plane curve (\cite[Section 4.4]{He98}). The flip $\beta$ is a wall-crossing of $\alpha$-stable pairs (Section \ref{ssec:moduliofpairs}). The other flip $\gamma$ initially came from the log minimal model program (MMP) of $\rM$, but it turns out that it is a wall-crossing in the sense of Bridgeland (Section \ref{ssec:Bridgeland}). The key technical result, which connects two different flips $\beta$ and $\gamma$ is that the `master' space of the flip $\gamma$ is indeed the moduli space of $(+)$-stable pairs $\rM^{+}$. 
\begin{theorem}[\protect{Theorem \ref{thm:mainprop}}]\label{thm:mainpropintro}
\begin{enumerate}
\item There is a flip $\gamma$ between $\rM$ and $\rQ$. 
\item The flip in (1) can be decomposed into a smooth blow-up and a smooth blow-down, 
\[
	\xymatrix{&\rM^+\ar[dr]^{q}\ar[ld]_{r}&
	\\ \rM\ar@{<-->}[rr]^{\gamma}&&\rQ}
\]
and the intermediate master space is $\rM^{+}$.
\end{enumerate}
\end{theorem}

To show this Theorem, first of all, we show that the projective bundle $\rQ$ over $\rN$ can be translated as a moduli space of Bridgeland stable objects with an appropriate stability condition (Lemma \ref{lem:QisBridgeland}). The computation of the wall-crossing locus is straightforward. A key observation is that this wall-crossing locus coincides with the blow-up center of $r : \rM^{+} \to \rM$. To show the existence of morphism $q: \rM^{+} \to \rQ$, we perform the elementary modification of pairs. The elementary modification (a sheaf version of stable reduction) is a useful trick to construct a morphism between two moduli spaces of sheaves or their variations, and there have been many concrete examples (for instance, \cite{CCK05} for vector bundles, \cite{CK11} for pure sheaves, \cite{Lo13} for Bridgeland stable objects). After the construction of $q$, by studying the local structure of its exceptional locus, we showed that $q$ is indeed a smooth blow-up. 

Now we go back to the original problem calculating the Chow ring of $\rM$. The Chow ring of $\rN$ was computed by Ellingsrud and Str\o mme (\cite{ES89}). Thus by carefully applying the projective bundle formula (\cite[Example 8.3.4]{Ful98}), the blow-up formula (\cite[Theorem 1 in Appendix]{Kee92}), and performing the blow-down computation (Section \ref{sec:ChowM}), we are able to obtain Theorem \ref{thm:ChowMintro}. 

Since the ring presentation is complicated, to reduce the number of relations as well as to find a simpler set of relations, it was indispensable to use a computer algebra system. We used Macaulay2 (\cite{M2}) to perform this calculation. We have posted all source codes we have used during the computation and their outputs at the second author's website:
\begin{center}
\url{http://www.hanbommoon.net/publications/chow-ring-m/}
\end{center}

\subsection{Stream of this paper}
In Section \ref{sec:relevantmoduli}, we introduce several relevant moduli spaces. Section \ref{sec:masterspace} is devoted to the proof of Theorem \ref{thm:mainprop} by using the elementary modification of pairs. In next three sections, we compute Chow rings of $\rN$, $\rM^{+}$, and $\rM$. Section \ref{sec:Cherneff} is for the computation of the total Chern class and some effective cycles on $\rM$. Finally, in the last section we evaluate Euler characteristics of some line bundles on $\rM$, which naturally appear on the strange duality conjecture on the plane.

\subsection*{Notation}
We work on the field $\CC$ of complex numbers. In this paper, the Chow rings and cohomology rings are always that with rational coefficients. $V$ is a fixed 3-dimensional complex vector space and if there is no further explanation, $\PP^2= \PP V$. Every point of an algebraic variety is a closed point. The projective bundle of a locally free sheaf $\cF$ over $X$ is defined by 
\[
\PP(\cF):=\proj(\mathrm{Sym}^{\bullet} (\cF^*))\lr X, 
\]
so $\PP(\cF)$ is the space of one-dimensional subspaces of $\cF$. 

Finally, we list some notation for important spaces in this paper. 
\begin{itemize}
\item $\rM := \rM(4,1) = \rM(\PP^{2}, 4m+1)$: The moduli space of stable sheaves $F$ with Hilbert polynomial $\chi(F(m))=4m+1$.
\item $\rM^{\alpha} := \rM^{\alpha}(\PP^{2}, 4m+1)$: The moduli space of $\alpha$-stable pairs $(s,F)$ with Hilbert polynomial $\chi(F(m))=4m+1$.
\item $\rN := \rN(3;2,3)$: The moduli space of quiver representations of $3$-Kronecker quiver with the dimension vector $(2, 3)$.
\item $\rQ:=\PP(\cU)$, where $\cU$ is a rank $12$ vector bundle over $\bN$ defined in Equation \eqref{eqn:exactseqforQ}.
\item $\bH(n):=\mathrm{Hilb}^n(\PP^2)$: The Hilbert scheme of $n$ points on $\PP^2$.
\end{itemize}

\subsection*{Acknowledgement}
We would like to thank Jinwon Choi, Young-Hoon Kiem, and David Swinarski for their valuable comments. David Swinarski kindly made a Macaulay2 code to compute a presentation of total Chern class in Proposition \ref{prop:Chernclass}. KC is partially supported by Korea NRF grant 2013R1A1A2006037.


\section{Relevant moduli spaces}\label{sec:relevantmoduli}

\subsection{Moduli spaces of stable pairs}\label{ssec:moduliofpairs}
Let us recall the notion of the moduli space of pairs (more generally, coherent systems) introduced by Le Potier (\cite{LP93b, He98}). Let $(X, H)$ be a smooth polarized projective variety. A pair $(s, F)$ consists of a coherent sheaf $F$ on $X$ and a nonzero section $\cO_{X} \stackrel{s}{\to} F$. Fix a positive rational number $\alpha$. A pair $(s,F)$ is called \emph{$\alpha$-semistable} if $F$ is pure and for any subsheaf $F'\subset  F$, the inequality
\[
	\frac{P(F')(m)+\delta\cdot\alpha}{r(F')} \le
	\frac{P(F)(m)+\alpha}{r(F)}
\]
holds for $m\gg 0$. Here $\delta=1$ if the section $s$ factors through $F'$ and $\delta=0$ otherwise. When the strict inequality holds, $(s,F)$ is called an \emph{$\alpha$-stable} pair. As in the case of ordinary sheaves, one can define Jordan-H\"older filtration and $S$-equivalent classes of pairs.

There exists a projective scheme $\rM^{\alpha}(X, P(m))$ parameterizing $S$-equivalence classes of $\alpha$-semistable pairs with Hilbert polynomial $P(m)$ (\cite[Theorem 4.12]{LP93b} and \cite[Theorem 2.6]{He98}). We denote by $\rM^{\infty}(X, P(m))$ when $\alpha$ is sufficiently large and $\rM^{+}(X, P(m))$ when $\alpha$ is sufficiently small. When $P(m)$ is linear, the space $\rM^{\infty}(X,P(m))$ is the space of stable pairs in the sense of the Pandharipande-Thomas (\cite{PT09}). By decreasing the stability parameter $\alpha$, we obtain a sequence of flips among the moduli spaces of $\alpha$-stable pairs:
\[
\rM^{\infty}(X,P(m))\dashleftarrow\dashrightarrow \rM^{\alpha}(X,P(m))\dashleftarrow\dashrightarrow \rM^{+}(X,P(m)).
\]
\begin{proposition}\label{prop:stablepairs}
\begin{enumerate}
\item  (\cite[Theorem 4.3]{He98}) The moduli space $\rM^{\alpha}(X, P(m))$ carries a universal family of $\alpha$-stable pairs, if $\alpha$ is generic. This includes two extremal cases of $\alpha = \infty, +$. 
\item There exists a a natural forgetful map to the moduli space of semistable sheaves
\[
	r:\rM^{+}(X, P(m))\longrightarrow 
	\rM(X, P(m)),
\]
which maps $(s,F)$ to $F$.
\item (\cite[Section 4.4]{He98}) If $(X, H) = (\PP^{2}, \cO_{\PP^{2}}(1))$, the moduli space $\rM^{\infty}(\PP^{2}, dm+\chi)$ of $\infty$-stable pairs is isomorphic to the relative Hilbert scheme of $\left(\chi - \frac{d(3-d)}{2}\right)$-points on the universal degree $d$ curves. 
\end{enumerate}
\end{proposition}

In this paper, we use a special case of $(X, H) = (\PP^{2}, \cO_{\PP^{2}}(1))$ and $P(m) = 4m + 1$. 

\begin{definition}
Let $\rM^{\alpha} := \rM^{\alpha}(\PP^{2}, 4m+1)$. 
\end{definition}
The sequence of flips for $\rM^{\alpha}$ was studied in \cite[Section 3.1]{CC12}. There is a single flip 
\[
	\xymatrix{\rM^{\infty} \ar@{<-->}[rr] \ar[rd] && \rM^{+}\ar[ld]\\
	& \rM^{3}},
\]
which can be described as a composition of a smooth blow-up and a smooth blow-down. 

For a general $F \in \rM$, $\dim \rH^{0}(F) = 1$. Therefore the forgetful map in Proposition \ref{prop:stablepairs} is a birational map. Indeed, $r : \rM^{+} \to \rM$ is a smooth blow-up along the Brill-Noether locus $\{F \in \rM\;|\; \dim \rH^{0}(F) = 2\}$ (\cite[Proposition 4.4]{CC12}). 

Finally, as in the case of an ordinary moduli space $\rM(d, \chi)$ of torsion sheaves, there is a Fitting map
\[
	\rM^{\alpha} \to |\cO_{\PP^{2}}(4)|
\]
which maps $(s, F)$ to the support of $F$ (cf. \cite[Section 2.2]{LP93a}). 

\subsection{Moduli spaces of Kronecker quiver representations}\label{sub:defq}
The moduli space of representations of a Kronecker quiver can be constructed as a GIT-quotient.  Let $\rN(3;d-2,d-1)$ the moduli space of quiver representations of $3$-Kronecker quiver
\[
\xymatrix{\bullet\ar@/^/[r] \ar@/_/[r]\ar[r]&\bullet\\}
\]
with dimension vector $(d-2, d-1)$. It can be identified with the space of isomorphism classes of stable sheaf homomorphisms
\begin{equation}\label{res1}
\cO_{\PP^2}(-2)^{\oplus d-2}\longrightarrow \cO_{\PP^2}(-1)^{\oplus d-1}
\end{equation}
up to the action of the automorphism group $G := \mathrm{GL}_{d-2}\times \mathrm{GL}_{d-1}/\CC^*$. Thus for two vector spaces $E$ and $F$ of dimension $d-2$ and $d-1$ respectively and $V^* = \rH^0(\cO_{\PP^2}(1))$, the GIT quotient $\rN(3;d-2,d-1) := \Hom(F,V^{*}\otimes E)\git_{L}G$ is a natural way to construct the moduli space. Indeed, with an appropriate linearization $L$, the GIT stability is equivalent to the stability of Kronecker's sense (For detail, see \cite{Kin94}).

\begin{proposition}[\protect{\cite{Kin94}}]
The space $\rN(3;d-2,d-1)$ is a smooth projective variety of dimension $(d-1)(d-2)$ and carries a universal family of quiver representations $\cF \to V^{*} \otimes \cE$ where $\cF$ and $\cE$ are two universal bundles of rank 2 and 3 respectively. 
\end{proposition}

Let $Z$ be a finite subscheme of $\PP^{2}$ of length $n := (d-1)(d-2)/2$. Since a resolution of a general ideal sheaf $I_Z(1)$ twisted by $\cO_{\PP^{2}}(1)$ is of the form \eqref{res1}, one can see that the moduli space $\rN(3;d-2,d-1)$ is birational to the Hilbert scheme $\bH(n)$ of $n$ points in $\PP^2$. 

In this paper, the relevant case is $d = 4$ (so $n = 3$). 
\begin{definition}
Let $\rN := \rN(3;2,3)$.
\end{definition}

The birational map between $\rH(3)$ and $\rN$ has a nice local structure. 
\begin{proposition}[\protect{\cite[Section 6]{LQZ03}}]\label{prop:hilbcon}
There exists a smooth blow-down morphism
\[
	t: \bH(3)\longrightarrow \rN.
\]
which contracts the divisor of three collinear points.
\end{proposition}

\begin{remark}\label{rem:setmapn}
\begin{enumerate}
\item If $Z$ is a triple of non-collinear points, then $t([Z])=I_Z(1)$. If $Z$ is collinear and supported on a fixed line $\ell \subset \PP^2$, then $t([Z])=E_\ell\in \PP\Ext_{\PP^2}^1(\cO_{\PP^{2}}, \cO_\ell(-2)) = \{\mathrm{pt}\}$ and  the image only depends on the line $\ell$ (\cite[Section 3.3]{DM11}). Thus the image of the locus of collinear points is naturally isomorphic to $\PP V^{*}$. 

\item There is a regular map $\bH(3) \to \mathrm{Gr}(3,\mathrm{Sym}^{2}V^{*})$, which is defined by $I_{Z} \mapsto \rH^{0}(I_{Z}(2))\subset \rH^0(\cO_{\PP^{2}}(2))=\mathrm{Sym}^{2}V^{*}$. Then the map $t$ is the restriction to the image (\cite[Equation (4.2)]{LQZ03}). In other words, $\rN$ is naturally embedded into $\mathrm{Gr}(3, \mathrm{Sym}^{2}V^{*})$. 
\end{enumerate}
\end{remark}

We may identify $\rN$ with the moduli space of (stable) sheaf homomorphisms
\[
	\cO_{\PP^{2}}(1)^{2} \to \cO_{\PP^{2}}(2)^{3}
\]
by tensoring $\cO_{\PP^{2}}(3)$. On $\rN \times \PP^{2}$, there is a universal morphism
\begin{equation}\label{eqn:universalmorphism}
	\phi : \pi_{1}^{*}\cF \otimes \pi_{2}^{*}\cO_{\PP^{2}}(1) 
	\to \pi_{1}^{*}\cE \otimes \pi_{2}^{*}\cO_{\PP^{2}}(2)
\end{equation}
where $\pi_{1} : \rN \times \PP^{2} \to \rN$ and $\pi_{2} : \rN \times \PP^{2} \to \PP^{2}$ are two projections. Let $\cU$ be the cokernel of
\[
	\pi_{1 *}\phi : \cF \otimes \rH^{0}(\cO_{\PP^{2}}(1)) =
	\pi_{1 *}(\pi_{1}^{*}\cF \otimes \pi_{2}^{*}\cO_{\PP^{2}}(1)) \to
	\pi_{1 *}(\pi_{1}^{*}\cE \otimes \cO_{\PP^{2}}(2)) = \cE \otimes
	\rH^{0}(\cO_{\PP^{2}}(2)).
\]
On the stable locus, $\phi$ is injective. Thus we have an exact sequence
\begin{equation}\label{eqn:exactseqforQ}
	0 \to \cF\otimes \rH^{0}(\cO_{\PP^{2}}(1)) \to
	\cE \otimes \rH^{0}(\cO_{\PP^{2}}(2)) \to \cU \to 0
\end{equation}
and $\cU$ is a rank 12 bundle.

\begin{definition}[\protect{\cite[Section 3.1.2]{DM11}}]\label{def:Q}
Let $\rQ := \PP(\cU)$. 
\end{definition}

\subsection{Moduli spaces of Bridgeland stable objects}\label{ssec:Bridgeland} 
After the initiation of the study of the birational geometry of Hilbert scheme $\rH(n)$ in terms of Bridgeland stability in \cite{ABCH13}, Mori's program for moduli spaces of sheaves has been generalized into pure one-dimensional sheaves cases by \cite{Woo13, BMW14}. The moduli space $\rM(d,\chi)$ is a Mori dream space (\cite[Corollary 3.6]{Woo13}), so the finiteness of rational contractions and the chamber decomposition of the effective cone are guaranteed. We briefly explain the relation between Bridgeland stability space and the chamber structure on the effective cone of $\rM(d, \chi)$. We provide a simplified definition for $\PP^{2}$ only. 

For $s \in \RR$ and $t > 0$, the potential function is given by
\[
	Z_{s,t}(E)=- \int_{\PP^2} e^{-(s+it)H} \cdot ch(E)
	= -\left(ch_{2}- sc_{1} + \frac{1}{2}r(s^{2}-t^{2})\right)
	+ it(c_{1}-rs).
\]
where $(r, c_{1}, ch_{2}) = (r(E), c_{1}(E),ch_{2}(E))$ and $H$ is the hyperplane class of $\PP^{2}$. The slope function $\mu_{s,t} : \mathrm{K}(\PP^{2}) \to \RR$ is
\[
	\mu_{s,t}(E)=
	-\frac{\mathrm{Re}(Z_{s,t}(E))}{\mathrm{Im}(Z_{s,t}(E))}
	= \frac{ch_{2} - sc_{1}+\frac{1}{2}r(s^{2}-t^{2})}{t(c_{1}-rs)}.
\]
As in the case of coherent sheaves, one can define a stability of objects in the heart $\cA_{s}$ of a certain $t$-structure (depending only on $s$) in $\mathrm{D}^b(\PP^2)$ by using the slope function $\mu_{s,t}$. The upper half plane $\HH = \{(s,t)\in \RR^{2}\;|\; t> 0\}$ can be regraded as the stability space (\cite{Woo13, BMW14}). Fix a topological type $v = (r, c_{1}, ch_{2}) \in \rA^{*}(\PP^{2})$. Let $\rM_{\mu_{s,t}}(v)$ be the moduli space of Brigdeland stable objects in $\cA_{s}$ with the topological type $v$. Then the stability space can be decomposed into finitely many chambers and $\rM_{\mu_{s,t}}(v)$ varies only if $(s,t)$ moves from a chamber to another chamber. Also if $t \gg 0$, $\rM_{\mu_{s,t}}(v) = \rM(v)$, the moduli space of ordinary semistable sheaves with topological type $v$ (\cite[Section 7]{Woo13}). 

From now, consider a special case of $v(F) = (0, d, \frac{-3d+2}{2})$ (and thus $\chi(F(m)) = dm+1$). Then $\rM(v) = \rM(d, 1)$. In this case the walls dividing chambers in $\HH$ are semi-circles with the common center $(-\frac{3d-2}{2d},0)$ (\cite[Lemma 6.9]{Woo13}). Note that the $\mathrm{Pic}(\rM(d, \chi))\cong \ZZ\oplus \ZZ$ for $d\geq 3$ (\cite[Theorem 1.1]{LP93a}). 

\begin{proposition}[\protect{\cite[Theorem 1.1]{BMW14}}]\label{prop:onetoone}
There is a one-to-one correspondence between
\begin{enumerate}
\item the chambers of the stability space $\HH$;
\item the birational models we obtain by running a directed MMP.
\end{enumerate}
\end{proposition}


\section{A master space between $\rM$ and $\rQ$ via a space of pairs}\label{sec:masterspace}

The goal of this section is to complete the following diagram. 
\begin{equation}\label{eqn:diagramSec3}
	\xymatrix{\rN&\bH(3)\ar[l]_{t}&\rM^{\infty}\ar@{=>}[l]
	\ar@{~>}[d]\ar@{<-->}[ld]_{\beta}\\
	\rQ \ar@{=>}[u]_{p}&\rM^{+}\ar[l]_{q} \ar[d]^{r}\ar@{~>}[r]
	&\rM^{3}\ar@{=>}[d]\\
	&\rM \ar@{<-->}[lu]^{\gamma} \ar@{=>}[r] 
	&|\cO_{\PP^{2}}(4)|}
\end{equation}
An ordinary arrow $\to$ is a smooth blow-up, a double arrow $\Rightarrow$ is a fibration with positive dimensional fiber, a wiggled arrow $\leadsto$ is a small contraction, and a dashed arrow $\dashrightarrow$ is a flip. Indeed, all maps, except $\gamma$ and $q$, are able to be explained in a broader theoretical context (Section \ref{sec:relevantmoduli}). For instance, $\rM^{\infty} \to \rH(3)$ is a projection, because $\rM^{\infty}$ is isomorphic to the relative Hilbert scheme of 3 points on the universal planar quartic (Proposition \ref{prop:stablepairs}). The flip $\beta$ is a wall-crossing of $\alpha$-stable pairs. The technical result, which we will show in this section, is the following. 

\begin{theorem}\label{thm:mainprop}
\begin{enumerate}
\item There is a flip $\gamma$ between $\rM$ and $\rQ$. 
\item The flip in (1) can be decomposed into a smooth blow-up and a smooth blow-down, 
\[
	\xymatrix{&\rM^+\ar[dr]^{q}\ar[ld]_{r}&
	\\ \rM\ar@{<-->}[rr]^{\gamma}&&\rQ}
\]
and the intermediate master space is $\rM^{+}$.
\end{enumerate}
\end{theorem}

\subsection{Review of Bridgeland wall crossing for $\rM$} 

We recall the result in \cite{DM11} on the description of moduli points in $\rM$ (\cite[Section 3.3]{DM11}). There are three types of objects in $\rM$:
\begin{enumerate}
\item $F = \cO_{C}(p+q+r)$, where $C$ is a quartic curve and $p, q, r$ are three non-collinear points on $C$;
\item $F$ fits into the non-split extension $0 \to \cO_{C} \to F \to \cO_{\ell} \to 0$, where $C$ is a cubic curve and $\ell$ is a line;
\item $F = I_{p, C}(1)$ where $C$ is a quartic curve and $p \in C$. In this case $\dim \rH^{0}(F) = 2$, otherwise $\dim \rH^{0}(F) = 1$. 
\end{enumerate}
Note that the \emph{Brill-Noether locus} of stable sheaves of the form in (3) is isomorphic to the universal quartic curve space $\cC_4:=\{(C,p)\;|\; p\in C\} \subset |\cO_{\PP^{2}}(4)| \times \PP^{2}$. 

The first part of Theorem \ref{thm:mainprop} follows from the lemma below.
\begin{lemma}\label{lem:QisBridgeland}
Let $\rM \dashrightarrow \rM_{\mu_{s,t}}(v)$ be the first Bridgeland wall-crossing. Then the main irreducible component of $\rM_{\mu_{s,t}}(v)$ is isomorphic to $\rQ$. The flipping locus (wall crossing locus) in $\rM$ is isomorphic to $\cC_4$, and that in $\rQ$ is $\mathrm{Fl}(V)$, the full flag variety of $V$.  
\end{lemma}

\begin{proof}
Let $\pi:\rM\lr |\cO_{\PP^{2}}(4)|$ be the Fitting map. Since $-K_{\rM}=12\pi^*\cO_{|\cO_{\PP^{2}}(4)|}(1)$ (\cite[Lemma 3.1]{Woo13}), $-K_{\rM}$ is nef. Also, because the Fitting map has positive dimensional fibers, $-K_{\rM}$ is an extremal ray of $\mathrm{Eff}(\rM)$. Thus all rational contractions appear on the one side of $-K_{\rM}$. Now by \cite[Proposition 2.2]{CC15}, $\rQ$ is a rational contraction of $\rM$, so it appears on the directed MMP of $\rM$. Also we know that $\rQ$ is the last birational model, because of the existence of the fibration $p : \rQ \to \rN$. From the numerical wall computation (for instance see \cite[Section 3]{CC15} and \cite[Section 6]{BMW14}), there is only one numerical wall when the radius of the semicircle is $R = \frac{7}{4}$. This proves the first statement. 

If we regard $\rM$ as a moduli space of Bridgeland stable objects, $F \in \cC_{4}$ fits into the exact triangle
\begin{equation}\label{wall1}
	\ses{I_p(1)}{F}{\cO_{\PP^{2}}(-3)[1]}
\end{equation}
where $p \in \PP^{2}$ (\cite[Lemma 2.10]{CC15}). After the wall-crossing, from the modification of complexes (\cite[Proposition 2.1]{Lo13}) the sub/quotient complexes of $F$ are interchanged. Thus we obtain 
\[
	\ses{\cO_{\PP^{2}}(-3)[1]}{F'}{I_p(1)}.
\]
For a fixed $p$, these semistable complexes $F'$ are parameterized by $\PP\Ext^1(I_p(1),\cO_{\PP^2}(-3)[1]) \cong \PP \rH^{0}(I_{p}(1))^*=\PP^1$. Therefore in $\rQ$, the flipping locus is $\PP \rH^{0}(I_{p}(1))^*$-bundle over $\PP V$, which is the full flag variety $\mathrm{Fl}(V)$.
\end{proof}

\subsection{Construction of the map $\rM^+\longrightarrow \rQ$}
In this section, we prove that the common blown-up space (so called the \emph{master} space) of $\rM$ and $\rQ$ is the moduli space $\rM^{+}$.

\begin{lemma}\label{lem:rigidmap}
There exists a dominant morphism
\[
	w: \rM^+\longrightarrow \rN
\]
which maps $(s,F)$ to the ideal sheaf of three non-collinear points or its supporting line.
\end{lemma}

\begin{proof}
By item (3) of Proposition \ref{prop:stablepairs}, $\rM^{\infty}$ is the relative Hilbert scheme of 3 points on the universal quartic. Moreover, it is a $\PP^{11}$-bundle over $\rH(3)$ (\cite[Lemma 2.3]{CC12}). By composing with $\rH(3)\lr \rN$ in Proposition \ref{prop:hilbcon}, we have a morphism $\rM^{\infty}\to\rN$.

On the other hand, there is a blown-up space $\widetilde{\rM^{\infty}}$ which dominates two spaces $\rM^{\infty}$ and $\rM^+$ (\cite[Theorem 3.3]{CC12}):
\begin{equation}
\xymatrix{&\widetilde{\rM^{\infty}}\ar[rd]\ar[ld]&\\
\rM^{\infty}\ar[d]\ar@{<--}[r]&\mathrm{flip}&\rM^+\ar@{-->}[lld]\ar@{<--}[l]\\
\bN&&}
\end{equation}
For $x := [(s, I_{p, C}(1))]\in \rM^{+}$ in the exceptional set of $r : \rM^{+} \to \rM$, we may think it as a collection of data $(p, C, \ell)$ where $p$ is a point, $C$ is a quartic curve, and $\ell$ is a line such that $p \in C \cap \ell$. Then the fiber over $x$ of $\widetilde{\rM^{\infty}} \to \rM^{+}$ is $\PP^{3}$, which parametrizes triples of points on a line passing through $p$. One can check that the composition map $\widetilde{\rM^{\infty}}\lr \rN$ is constant along the $\PP^3$, because $\rN$ forgets the configuration of collinear points and remembers the line containing them only. Hence by the rigidity lemma, $\widetilde{\rM^{\infty}}\lr \rN$ factors through $\rM^{+}$ and we obtain the map $w: \rM^{+}\lr \rN$. It is dominant since $\rM^{\infty} \to \rN$ is. 
\end{proof}

Let $r:\rM^+\lr \rM$ be the forgetful map in (2) of Proposition \ref{prop:stablepairs}. Set theoretically, the map $w: \rM^+\rightarrow \rN$ can be described as the following. 
\begin{enumerate}
\item If $(s,F)\in r^{-1}(\{F\;|\;F = \cO_{C}(p+q+r)\})$, then $w((s,F))=I_{\{p,q,r\}}(2)$.
\item If $(s,F)\in r^{-1}(\{F\;|\;\ses{\cO_C}{F}{\cO_\ell}\})$, then $w((s,F))=E_{\ell}$ (see Remark \ref{rem:setmapn} for the notation).
\item If $(s,F)\in r^{-1}(\cC_4 = \{F\;|\; F = I_{p, C}(1)\})$, then $w((s,F))=E_{\ell}$ for the line $\ell$ determined by $p$ and $s\in \rH^0(F)$.
\end{enumerate}

Next, we show that the map $w$ factors through $\rQ$ in Definition \ref{def:Q}, which is a $\PP^{11}$-bundle over $\rN$. This proves item (2) of Theorem \ref{thm:mainprop}.

\begin{proposition}\label{prop:mainprop}
There is a smooth divisorial contraction $q$
\[
	\xymatrix{&\rQ := \PP(\cU)\ar[d]^{p}\\ 
	\rM^+\ar[r]^w\ar@{-->}[ru]^{\exists \;q}&\bN}
\]
which lifts the morphism $w$ in Lemma \ref{lem:rigidmap}.
\end{proposition}

By Proposition \cite[Proposition II.7.12]{Har77}, giving a lifting of the map $w$ is equivalent to giving a surjective homomorphism $w^*\cU^*\lr \cL\lr 0$ for some $\cL\in \mathrm{Pic}(\rM^+)$, or equivalently, an injective \emph{bundle} morphism (fiberwise injective morphism) $0 \to \cL^{*} \to w^{*}\cU$. Note that in our convention, $\PP(\cU) = \proj (\mathrm{Sym}^{\bullet}\cU^{*})$. 

The proof of Proposition \ref{prop:mainprop} is divided into several steps. Before proving the proposition, let us recall the following.

\begin{remark}
The notion of a family of pairs is delicate. One reason is that the global section functor, more generally the direct image functor, does not behave well. Thus to define a family formally, we use the dual of a pair. By definition (\cite[Definition 4.8]{LP93b}), a family of pairs over a scheme $S$ is a pair $(\cL, \cF)$ where
\begin{itemize}
\item $\cL \in \mathrm{Pic}(S)$;
\item $\cF\in \mathrm{Coh}(S \times \PP^2)$ is a flat family of pure sheaves;
\item a surjective homomorphism 
\begin{equation}\label{eqn:familyofpairs}
\cExt^2_{\pi}(\cF, \omega_{\pi})\twoheadrightarrow \cL
\end{equation}
on $S$ where $\pi : S \times \PP^{2} \to \PP^{2}$ is the projection and $\omega_{\pi}$ is the relatively dualizing sheaf. 
\end{itemize}
\end{remark}

Let $(\cL,\cF)$ be the universal pair (the existence is guaranteed by Proposition \ref{prop:stablepairs}) on $\rM^+\times \PP^2$, where the $\pi:\rM^+\times \PP^2\lr \rM^+$ be the projection map. By \cite[Corollary 8.18]{LP93b}, there is a spectral sequence $\mathrm{E}_{2}^{p,q} = \cExt^{p}(\cExt_{\pi}^{2-q}(\cG, \cO), \cO) \Rightarrow R^{p+q}\pi_{*}(\cG\otimes \omega_{\pi})$. In particular, if $p = q = 0$, $\cHom(\cExt_{\pi}^{2}(\cG, \cO), \cO) \cong \pi_{*}(\cG \otimes \omega_{\pi})$. By applying $\cHom(-,\cO)$ functor to \eqref{eqn:familyofpairs}, we obtain an injection
\[
	0 \to \cL^{*} \to \cHom(\cExt^{2}_{\pi}(\cF, \omega_{\pi}), \cO) 
	\cong \pi_{*}(\cF \otimes \omega_{\pi}^{*}\otimes \omega_{\pi}) 
	\cong \pi_{*}(\cF) \cong \cExt_{\pi}^1(\cExt^1(\cF,\cO),\cO).
\]
Note that $\cExt^{1}(\cF, \cO)$ is a torsion sheaf, and since $\cF$ is fiberwisely Cohen-Macaulay, $\cExt^1(\cExt^1(\cF,\cO),\cO)\cong \cF$ by the proof of \cite[Lemma 5.9]{LP93b}. The last isomorphism is obtained from the Grothendieck spectral sequence for $\pi_{*}\circ \cHom$. By the local-to-global spectral sequence for $\cHom_{\pi}$ and 
\[
\begin{split}
\Hom(\cL^{*}, \cExt_{\pi}^{1}(\cExt^1(\cF,\cO), \cO)) \cong &\;\rH^{0}(\cExt_{\pi}^{1}(\cExt^1(\cF,\cO), \cO)\otimes \cL)\\ \cong&\; \rH^{0}(\cExt_{\pi}^{1}(\cExt^1(\cF,\cO), \pi^{*}\cL))) \cong \Ext^{1}(\cExt^1(\cF,\cO), \pi^{*}\cL),
\end{split}
\]
there is an element
\[
e\in \Ext^1(\cExt^1(\cF, \cO),{\pi}^*\cL),
\]
which provides an exact sequence
\[
\ses{{\pi}^*\cL}{\cE}{\cExt^1(\cF, \cO)}
\]
on $\rM^+\times \PP^2$. Since $\cExt^1(\cF,\cO)$ is a flat family of sheaves, $\cE$ also is. By taking $\cHom_{\pi}(-,\omega_{\pi})$, we have a surjection
\[
\cExt_{\pi}^2(\cE,\omega_{\pi})\rightarrow \cExt_{\pi}^{2}(\pi^{*}\cL, \omega_{\pi}) \cong \cL^{*}\lr \cExt^3(\cExt^1(\cF, \cO), \omega_{\pi})=0
\]
because $\cL$ is a line bundle. This implies the existence of a flat family of pairs $(\cL^*,\cE)$ on $\rM^+\times \PP^2\lr \rM^+$. 

The above operation can be described fiberwisely as the following way. Consider the pair $(s^*,G)$ defined by
\begin{equation}\label{eq-1}
\ses{(s^{*})\otimes \cO_{\PP^{2}}}{G}{F^{D}(3)}
\end{equation}
which it is given by a non-zero section $ s\in \rH^0(F)\cong\rH^1(F^D)^{*} \cong\Ext^1(F^D(3),(s^{*})\otimes \cO_{\PP^2})$. Here $F^D:=\cExt^1(F, \omega_{\PP^2})$ (cf. \cite[Theorem 5.5]{LP93a}). 

\begin{definition}\label{def:EandW}
\begin{enumerate}
\item Let $E$ be the exceptional divisor of the blow-up $r:\rM^+\lr \rM$.
\item Let $W\subset \rM^+$ be the locus of pairs $(s,F)$ such that $F\in \Ext^1(\cO_{\ell},\cO_C)$ for a cubic $C$ and a line $\ell$.  
\end{enumerate}
\end{definition}

Then $E$ is isomorphic to a $\PP^1$-bundle over $\cC_4$ (\cite[Proposition 4.4]{CC12}). 

\begin{lemma}\label{lem:destab}
The correspondence $(s, F) \mapsto (s^{*},G)$ defines a dominant rational map $\rM^{+} \dashrightarrow \rQ$, which is regular on $\rM^{+} \setminus E$. 
\end{lemma}

\begin{proof}
Since we already have a relative construction of pairs, it suffices to describe the extensions $(s^*,G)$ set theoretically. If $(s,F)\in \rM^+-E \cup W$, then $F\cong \cO_{C}(Z) \cong I_{Z,C}^D(-1)$ for some non-collinear three points $Z$ in a quartic curve $C$. Then, $F^D(3)\cong I_{Z,C}(4)$. Hence, from the structure sequence $\ses{\cO_{\PP^{2}}(-4)\cong I_{C,\PP^2}}{I_{Z,\PP^2}}{I_{Z,C}}$, we obtain $G=I_{Z,\PP^2}(4)$ since $\Ext^{1}(F^{D}(3), \cO_{\PP^{2}}) \cong \rH^0(F)\cong\CC$. Hence we have an element $(s^*, G)\in \rQ$.

Suppose that $(s,F)\in W \setminus E$. Note that $F$ fits into a non-split extension $\ses{\cO_C}{F}{\cO_\ell}$. Apply $\cHom(-,\omega_{\PP^{2}})$, then we have $\ses{\cO_{\ell}(1)}{F^{D}(3)}{\cO_{C}(3)}$. Since $\Ext^1(\cO_C(3),\cO_{\PP^{2}})\cong \Ext^1(F^D(3),\cO_{\PP^{2}})\cong\CC$, the sheaf $G$ is given by the pulling-back:
\begin{equation}\label{eqn:ladder1}
\xymatrix{0\ar[r]&\cO_{\PP^{2}}\ar[r]\ar@{=}[d]
&\cO_{\PP^{2}}(3)\ar[r]&\cO_C(3)\ar[r]&0\\
0\ar[r]&\cO_{\PP^{2}}\ar[r]&G\ar[r]\ar@{-->}[u]&F^D(3)\ar[u]\ar[r]&0
}
\end{equation}
By applying the snake lemma to \eqref{eqn:ladder1}, we conclude that the unique non-split extension $G$ lies on the short exact sequence $\ses{\cO_\ell(1)}{G}{\cO_{\PP^{2}}(3)}$. Hence we have an element $(s^*, G)\in \rQ$.

Finally, suppose that $(s,F)\in E$, so $F$ fits into an exact sequence $\ses{I_{C}(1) \cong \cO_{\PP^{2}}(-3)}{I_{p}(1)}{F}$. By taking $\cHom(-,\cO_{\PP^{2}})$, we obtain $0 \to \cO_{\PP^{2}}(-1) \to \cO_{\PP^{2}}(3) \to F^{D}(3) \to \CC_{p} \to 0$. Note that $\Hom (\cO_{\PP^{2}}(3), F^D(3))\cong\rH^0(F^D) \cong\rH^1(F)^{*}\cong \CC$. Also, $\Hom (\cO_{\PP^{2}}(3), G) \cong \Hom (\cO_{\PP^{2}}(3), F^D(3))$ since $\rH^i(\cO_{\PP^{2}}(-3))=0$ for $i=0,1$. Hence a non-zero homomorphism $\cO_{\PP^{2}}(3)\lr F^D(3)$ is lifted to a nonzero morphism $j:\cO_{\PP^{2}}(3) \lr G$. By the snake lemma, we obtain the following commutative diagram:
\[
\xymatrix{
&&& 0\ar[d]\\
&& 0 \ar[r] \ar[d]  & \cO_{\PP^{2}}(-1) \ar[d]
\ar@{-}`r[d]`[d][dll]-<1.35cm,1.515cm>|!{"2,3";"3,3"}\hole 
& \\
& 0 \ar[r] \ar[d] & \cO_{\PP^{2}}(3) \ar[d]^{j} \ar@{=}[r]
& \cO_{\PP^{2}}(3) \ar[d] \ar[r]& 0 \\
0 \ar[r]  & (s^{*})\otimes \cO_{\PP^{2}} \ar@{=}[d]\ar[r]
\ar`_l[l]+<0.8cm,0cm>`[d]`[d][d]
& G \ar[r]\ar[d]   & F^{D}(3) \ar[d] \ar[r]& 0   \\
& (s^{*})\otimes \cO_{\PP^{2}}\ar[d] \ar[r] & H\ar[d] \ar[r] 
& \CC_{p}\ar[d] &   \\
&0&0&0}
\]
In particular, $j$ is injective. But in this case $G$ does not have a resolution $\ses{\cO_{\PP^{2}}(1)^{2}}{\cO_{\PP^{2}}(2)^{3}}{G}$, because we obtain $\Hom(\cO_{\PP^{2}}(3), G) =0$ from $\Hom(\cO_{\PP^{2}}(3), \cO_{\PP^{2}}(2)) = \Ext^{1}(\cO_{\PP^{2}}(3),\cO_{\PP^{2}}(1)) = 0$. This implies $G(-2)\notin \rN$, hence the correspondence is not well-defined on $E$ (cf. (1) of Remark \ref{rem:setmapn}).
\end{proof}

The family of unstable pairs parameterized by $E$ in the proof of Lemma \ref{lem:destab} has a flat family of destabilizing pairs.

\begin{lemma}
Under the same notation in Lemma \ref{lem:destab}, there exists an exact sequence
\begin{equation}\label{lem:excep}
\ses{(0, K)}{(\cL^*|_{E},\cE|_{E\times \PP^2} )}{(\cL^{''},\cO_{\cZ}(1))}
\end{equation}
where $\cZ$ is the universal subscheme of lines in $(\PP^2)^*\times\PP^2$ and $K_{\{m\}\times \PP^2}\cong \cO_{\PP^{2}}(3)$ for $m\in E$.
\end{lemma}

\begin{proof}
When $m=[(s,F)]\in E$, then the section $s^* \in \rH^0(G)$ for $G=\cE|_{\{d\}\times \PP^2}$ does not comes from $\rH^0(\cO_{\PP^{2}}(3))$. If it is, then the cokernel of an injection $\cO_{\PP^{2}}\subset \cO_{\PP^{2}}(3)$ is isomorphic to the $\cO_C(3)$ for some cubic $C$. It violates the stability of $F^D(3)$. By combining with Lemma \ref{lem:destab}, the pair $(s^*,G)$ in the exact sequence \eqref{eq-1} fits into the exact sequence
\[
0\lr(0,\cO_{\PP^{2}}(3))\lr (s^*, G)\lr (s^*, \cO_{\ell}(1))\lr 0.
\]
This construction can be relativized and we obtain the destabilizing sub-pair $(0, K)$ and the quotient-pair $(\cL'',\cO_{\cZ}(1))$. 
\end{proof}

Now, to extend the rational map in Lemma \ref{lem:destab} to the entire $\rM^{+}$, we perform the elementary modification of pairs (cf. \cite[Section 2.2]{CC12}). Let 
\[
(\cL',\cE'):=\mathrm{Ker}( (\cL^*, \cE)\twoheadrightarrow  (\cL^*|_{E},\cE|_{E\times \PP^2} )\twoheadrightarrow (\cL^{''},\cO_{\cZ}(1))).
\]
be the kernel of the composition of the restriction map and the surjection in \eqref{lem:excep}.

\begin{lemma}\label{lem:modified}
For a point $m = [(s, F = I_{p,C}(1))] \in E$, the modified pair $(\cL',\cE')|_{\{m\}\times \PP^2}$ parameterized by $m$ fits into a non-split short exact sequence
\begin{equation}\label{eqn:modification}
\ses{(s',\cO_\ell(1))}{(s',\cE'|_{\{m\}\times \PP^2})}{(0,\cO_{\PP^{2}}(3))}.
\end{equation}
\end{lemma}

\begin{proof}
Note that $\Ext^1((0,\cO_{\PP^{2}}(3),(s',\cO_{\ell}(1)))\cong \rH^1(\cO_{\ell}(-2))=\rH^0(\cO_{\ell})^*$ (see \cite[Corollary 1.6]{He98} for the idea of computation). Performing a modification of pairs affects as the interchange of the sub/quotient pairs. 

It remains to show that \eqref{eqn:modification} is non-split. We claim that the normal bundle $\cN_{E/\rM^{+}}$ at $m$ is canonically isomorphic to $\rH^0(\cO_{\ell})^*$. Then the element $m$ corresponds to the projective equivalent class of nonzero elements in $\rH^{0}(\cO_{\ell})^{*} \cong \Ext^1((0,\cO_{\PP^{2}}(3),(s',\cO_{\ell}(1)))$, it is non-split.

The stable pair $(s,F)=(s,I_{p,C}(1))$ fits into an exact sequence 
\[
\ses{(0,\cO_{\PP^{2}}(-3))}{(s,I_{p}(1))}{(s,F)}.
\]
Thus there is a long exact sequence 
\[
\begin{split}
0&\;\lr \Ext^0((0,\cO_{\PP^{2}}(-3)), (s,F))\lr \Ext^1((s,F),(s,F))\\
&\;\lr \Ext^1((s,I_p(1)),(s,F))\lr \Ext^1((0,\cO_{\PP^{2}}(-3)), (s,F))\lr 0.
\end{split}
\]
Then $\Ext^0((0,\cO_{\PP^{2}}(-3)), (s,F))\cong \rH^0(F(3))\cong\ker(\rH^0(\cO_C(4))\to\rH^0(\CC_p))\cong\CC^{13}$. This is the deformation space of quartic curves while fixing the point $p$. The second term $\Ext^{1}((s, F), (s,F))$ is the tangent space of $\rM^{+}$ at $m$ (\cite[Theorem 3.12]{He98}). For the third term, consider
\[
\begin{split}
0=\Ext^1((s,I_p(1)), &(0,\cO_{\PP^{2}}(-3)))\lr \Ext^1((s,I_p(1)), (s,I_p(1))) \lr \Ext^1((s,I_p(1)),(s,F))\\
&\lr \Ext^2((s,I_p(1)), (0,\cO_{\PP^{2}}(-3)))\lr \Ext^2((s,I_p(1)), (s,I_p(1)))=0.
\end{split}
\]
The second term $\Ext^1((s,I_p(1)), (s,I_p(1)))\cong\CC^3$ is the direct sum of the deformation space of $p$ and the deformation space of the choice of a section $s$. Hence the normal space to $E$ at $m = [(s,F)]$ is isomorphic to the space $\Ext^2((s,I_p(1)), (0,\cO_{\PP^{2}}(-3)))$. From \cite[Corollary 1.6]{He98} again, this is isomorphic to the kernel of 
\[
\Ext^2(I_p(1), \cO_{\PP^{2}}(-3))\twoheadrightarrow \rH^2(\cO_{\PP^{2}}(-3)).
\]
By taking the Serre duality, we prove the claim.
\end{proof}

By definition, the modified pair $(\cL',\cE')$ provides a natural surjection
\begin{equation}\label{eq-3}
\cExt_{\pi}^2(\cE',\omega_{\pi})\twoheadrightarrow \cL'
\end{equation}
on $\rM^+$. It is straightforward to check that $\cExt_{\pi}^2(\cE',\omega_{\pi})$ has rank 12 at each fiber, thus it is locally free.

\begin{proof}[Proof of Proposition \ref{prop:mainprop}]

We claim that there exists a surjection
\[
	w^*\cU^*\lr \cL' \lr 0
\]
up to a twisting with a line bundle on $\rM^+$. Then there is a morphism $\rM^+\lr \rQ$. Because it is an extension of a dominant rational map between irreducible projective varieties, it is surjective. 

Consider the following commutative diagram
\[
\xymatrix{\rM^+\times \PP^2 \ar[r]_{w':=w\times \mathrm{id}}\ar[d]_{\pi}& \rN\times \PP^2\ar[d]^{\pi}\\
\rM^+ \ar[r]^{w}&\rN.}
\]
Recall that $\cU = \pi_*(\cE_0)$ where $\cE_0=\mathrm{coker}(\phi)$ is the twisted universal quotient on $\rN\times \PP^2$. Also $\cE_0$ is flat over $\rN$ (Section \ref{sub:defq}). By its construction of $w$, $\cE'|_{\{m\}\times \PP^2}\cong {w'}^{*}\cE_0|_{\{m\}\times \PP^2}$ restricted to a point $m\in \rM^+$. Hence the universal property of $\rN$ (as a quiver representation space \cite[Proposition 5.6]{Kin94}) tells us that ${w'}^{*}\cE_0\cong \cE'$ up to a twisting of a line bundle on $\rM^+$. The base change property implies that there exists a natural isomorphism (up to a twisting by a line bundle)
\[
w^*\cU=w^*(\pi_*\cE_0)\cong \pi_*({w'}^*\cE_0)=\pi_* \cE'
\]
because $\cU$ is locally free (\cite[Proposition 4.37]{FGetal05}).
By \cite[Corollary 8.19]{LP93b}, $\pi_* \cE'\cong \cExt_\pi^2(\cE',\omega_{\pi})^*$. Hence by \eqref{eq-3}, we have
\[
w^*\cU^*\cong (w^*\cU)^*\cong (\pi_* (\cE'))^*\cong \cExt_{\pi}^2(\cE',\omega_{\pi}) \twoheadrightarrow \cL'.
\]
Therefore we obtain a morphism $q : \rM^{+} \to \rQ$. 

Note that both $\rM^{+}$ and $\rQ$ are smooth and $\rho(\rM^{+}) = \rho(\rQ) + 1$. So the map $q: \rM^{+} \to \rQ$ is a divisorial contraction. As explained in the proof of Lemma \ref{lem:rigidmap}, the exceptional divisor $E$ is a $\PP^{13}$-bundle of quartics containing $p$, over the full flag variety $\mathrm{Fl}(V) \subset \rQ$. Also, by the proof of Lemma \ref{lem:modified}, the modified pair does not depend on the choice of a quartic curve. Therefore there is a commutative diagram
\[
	\xymatrix{&E=\mathrm{Fl}(V) \times_{\PP V}\cC_{4}
	\ar[ld]_{t} \ar[rd]^{s}\\
	\PP V^{*}\times \PP V \supset \mathrm{Fl}(V) \ar[rd]^{u}&&
	\cC_{4} \subset \PP V \times |\cO_{\PP^{2}}(4)| \ar[ld]_{v}\\
	&\PP V.}
\]
The smoothness of the blowing-down map $q$ comes from Fujiki-Nakano criterion. From Equation \eqref{wall1}, the sheaf $F\in \cC_4$ is parameterized by $\PP^{13}=\PP\Ext^1(\cO_{\PP^{2}}(-3)[1], I_p(1))$-bundle over $\PP^2$. As analyzing the tangent space $T_F\rM=\Ext^1(F,F)$ by using Equation \eqref{wall1} (which is similar to the proof of \cite[Lemma 3.4]{CC12}), one can see that the normal space along the fiber $\PP^{13}$ is isomorphic to
\[
\cN_{\cC_4/\rM}|_{v^{-1}(\{p\})}\cong 
\Ext^1(I_{p}(1),\cO_{\PP^{2}}(-3)[1])\otimes \cO_{\PP^{13}}(-1).
\]
Hence the normal bundle $\cN_{E/\rM^{+}}$, when it is restricted to the fiber of $p \in \PP V$, is isomorphic to $ \cO_{\PP^{13}\times \PP^1}(-1,-1)$. Therefore the divisorial contraction $q$ is a smooth blow-down by \cite{FN71}.
\end{proof}

\begin{remark}
Contrary to the ordinary wall-crossing of moduli spaces of pairs, the modified flat family $\cE'$ is not a family of pure sheaves anymore. A similar phenomenon appears if we replace $\PP^2$ by the quadric surface $\PP^1\times \PP^1$. This case will be investigated in a forthcoming paper.
\end{remark}


\section{Chow ring of $\rN$}\label{sec:ChowofN}

The moduli space of $\rN = \rN(3;2,3)$ is a GIT quotient of a projective space (Section \ref{sub:defq}). Its topological properties have been studied in several papers, including \cite{ES89, Tjo98, Fra15}. In this section, we compute the Chow ring of $\rN$ by following the computation in \cite{ES89}.

The Poincar\'e polynomial of $\rN$ can be computed by several different methods, for instance Kirwan's recursive algorithm described in \cite{Kir84}.

\begin{proposition}[\protect{\cite[Theorems 4.4, 7.8]{ES89}}]\label{prop:ppolyN}
\begin{enumerate}
\item The Chow ring and the cohomology ring of $\rN$ are isomorphic. 
\item 
\[
	P_{t}(\rN) := \sum_{i\ge 0}\rank \rA_{i}(\rN)t^{i} = 
	1 + t + 3t^{2} + 3t^{3} + 3t^{4} + t^{5} + t^{6}.
\]
\end{enumerate}
\end{proposition}

Let $E$ and $F$ be two vector spaces of dimension 3 and 2 respectively. The moduli space $\rN$ is constructed as the GIT quotient $\Hom(F, V^* \otimes E)\git G$ where $G = \GL(F) \times \GL(E)/\CC^{*}$. Let $T \cong (\CC^{*})^{4}$ be the diagonal maximal torus of $G$.

\begin{theorem}[\protect{\cite[Theorem 6.9, (i)]{ES89}}]
\[
\begin{split}
	\rA^{*}(\Hom(F, V^* \otimes E)\git T) \cong\; &
	\QQ[\beta_{1}, \beta_{2}, \beta_{3}, \delta_{1}, \delta_{2}]\\
	&/\langle
	\beta_{1}+\beta_{2}+\beta_{3}-\delta_{1}-\delta_{2},
	(\beta_{i}-\delta_{k})^{3}(\beta_{j}-\delta_{k})^{3},
	(\beta_{i}-\delta_{1})^{3}(\beta_{i}-\delta_{2})^{3}\rangle
\end{split}
\]
for $1 \le i < j \le 3$ and $k = 1, 2$. All $\beta_{i}$, $\delta_{k}$ are of degree 1.
\end{theorem}

The Chow ring of $\rN = \Hom(F, V^* \otimes E)\git G$ can be described in terms of the Weyl group $W$ action on $\Hom(F, V^* \otimes E)\git T$. Note that $W \cong S_{3} \times S_{2}$ and $W$ acts on $\{\beta_{i}, \delta_{k}\}$ as the permutation group.

Let $b_{i}$ (resp. $d_{k}$) be the elementary symmetric polynomials in $\beta_{i}$ (resp. $\delta_{k}$). There is a map
\[
	\varphi : \QQ[\beta_{1}, \beta_{2}, \beta_{3}, \delta_{1}, \delta_{2}]
	\to \QQ[b_{1}, b_{2}, b_{3}, d_{1}, d_{2}]
\]
defined by $\varphi(r) = \Delta^{-1}\sum_{w \in W}\mathrm{sign}(w)w(r)$, where
\[
	\Delta = (\beta_{1}-\beta_{2})(\beta_{1}-\beta_{3})
	(\beta_{2}-\beta_{3})(\delta_{1}-\delta_{2}).
\]
Let $I = \langle \beta_{1}+\beta_{2}+\beta_{3}-\delta_{1}-\delta_{2}, (\beta_{i}-\delta_{k})^{3}(\beta_{j}-\delta_{k})^{3}, (\beta_{i}-\delta_{1})^{3}(\beta_{i}-\delta_{2})^{3}\rangle$ be the ideal of relations.

\begin{theorem}[\protect{\cite[Theorem 6.9, (ii)]{ES89}}]\label{thm:cohomologyofNgeneral}
\[
	\rA^{*}(\rN) = \rA^{*}(\Hom(F, V^* \otimes E)\git G) \cong
	\QQ[b_{1}, b_{2}, b_{3}, d_{1}, d_{2}]/\varphi(I).
\]
\end{theorem}

To compute this ring explicitly, we need to compute $\varphi(I)$. Note that $\varphi$ is a map of degree $(-4)$ and for every anti-invariant $\Delta r $, $\varphi(\Delta r) = |\cW|r = 12r$. More precisely, $\varphi$ defines an isomorphism
\[
	\varphi : \QQ[\beta_{1}, \beta_{2}, \beta_{3},
	\gamma_{1}, \gamma_{2}]^{a} \to
	\QQ[b_{1}, b_{2}, b_{3}, d_{1}, d_{2}]
\]
of $\QQ[b_{1}, b_{2}, b_{3}, d_{1}, d_{2}]$-modules, and the domain is generated by $\Delta$. Thus to find a set of generators of $\varphi(I)$, it suffices to find a set of generators in $I^{a} := I \cap \QQ[\beta_{1}, \beta_{2}, \beta_{3}, \delta_{1}, \delta_{2}]^{a}$. The following elements are elements in $I^{a}$. For notational simplicity, we set $\beta_{i} := \beta_{i \!\!\mod 3}$ for $i > 3$.
\[
	\Delta (\beta_{1}+\beta_{2}+\beta_{3}-\delta_{1}-\delta_{2}),
\]
\[
	\sum_{i=1}^{3}(\beta_{i}-\beta_{i+1})(\beta_{i}-\delta_{1})^{3}
	(\beta_{i+1}-\delta_{1})^{3}
	- \sum_{i=1}^{3}(\beta_{i}-\beta_{i+1})(\beta_{i}-\delta_{2})^{3}
	(\beta_{i+1}-\delta_{2})^{3},
\]
\[
	\sum_{i=1}^{3}\delta_{1}(\beta_{i}-\beta_{i+1})
	(\beta_{i}-\delta_{1})^{3}(\beta_{i+1}-\delta_{1})^{3}
	- \sum_{i=1}^{3}\delta_{2}(\beta_{i}-\beta_{i+1})
	(\beta_{i}-\delta_{2})^{3}(\beta_{i+1}-\delta_{2})^{3},
\]
\[
	\sum_{i=1}^{3}(\beta_{i}^{2}-\beta_{i+1}^{2})
	(\beta_{i}-\delta_{1})^{3}(\beta_{i+1}-\delta_{1})^{3}
	- \sum_{i=1}^{3}(\beta_{i}^{2}-\beta_{i+1}^{2})
	(\beta_{i}-\delta_{2})^{3}(\beta_{i+1}-\delta_{2})^{3},
\]
\[
	\sum_{i=1}^{3}\delta_{1}(\beta_{i}^{2}-\beta_{i+1}^{2})
	(\beta_{i}-\delta_{1})^{3}(\beta_{i+1}-\delta_{1})^{3}
	- \sum_{i=1}^{3}\delta_{2}(\beta_{i}^{2}-\beta_{i+1}^{2})
	(\beta_{i}-\delta_{2})^{3}(\beta_{i+1}-\delta_{2})^{3},
\]
\[
	\sum_{i=1}^{3}(\beta_{i}-\beta_{i+1})(\beta_{i+2}-\delta_{1})^{3}
	(\beta_{i+2}-\delta_{2})^{3}(\delta_{1}-\delta_{2}),
\]
\[
	(\beta_{1}-\beta_{2})(\beta_{2}-\beta_{3})(\beta_{3}-\beta_{1})
	\left(\sum_{i=1}^{3}
	(\beta_{i}-\delta_{1})^{3}(\beta_{i+1}-\delta_{1})^{3} - 
	\sum_{i=1}^{3}
	(\beta_{i}-\delta_{2})^{3}(\beta_{i+1}-\delta_{2})^{3}\right).
\]

For each element $\alpha \in \QQ[\beta_{1}, \beta_{2}, \beta_{3}, \delta_{1}, \delta_{2}]$, by using a computer algebra system, we can compute $\varphi(\alpha)/|W|$. Note that the image of the first element is $b_{1}-d_{1}$. Thus we can identify $b_{1}$ and $d_{1}$. After this identification, we obtain the following elements:
\[
	r_{1} := -b_{1}^{3} + 4b_{1}d_{2} - 3b_{3},
	r_{2} := -b_{1}^{4} + 5b_{1}^{2}d_{2} - 2b_{1}b_{3} - b_{2}^{2}
	+ 3b_{2}d_{2} - 6d_{2}^{2},
\]
\[
	r_{3} := -b_{1}^{2}b_{2} + b_{2}d_{2} + 3d_{2}^{2},
	r_{4} := -b_{1}^{3}b_{2} +b_{1}^{2}b_{3} + 5b_{1}b_{2}d_{2}
	- b_{1}b_{2}^{2} + b_{2}b_{3} - 6b_{3}d_{2},
\]
\[
	r_{5}:= b_{1}b_{3} - b_{2}^{2} + 3b_{2}d_{2} - 3d_{2}^{2},
	r_{6}:= b_{1}^{5} - 3b_{1}^{3}b_{2} - 4b_{1}^{3}d_{2}
	+ 6b_{1}^{2}b_{3} + 9b_{1}b_{2}d_{2} + 3b_{1}d_{2}^{2}
	+ 3b_{2}b_{3} - 21b_{3}d_{2}.
\]

\begin{proposition}\label{cohnet}
\[
	\rA^{*}(\rN) \cong \QQ[b_{1}, b_{2}, d_{2}]/
	\langle b_1^2 d_2-3 d_2^2,b_1^2 b_2-b_2 d_2-3 d_2^2,
	b_1^4+3b_2^2-9b_2 d_2-3d_2^2,
	2b_1 b_2 d_2-3 b_1 d_2^2,
	 3b_1 b_2^2-7b_1 d_2^2\rangle
\]
\end{proposition}

\begin{proof}
Consider the ideal $J$ of $\QQ[b_{1}, b_{2}, b_{3}, d_{2}]$ generated by $r_{i}$ for $1 \le i \le 6$. It is a subideal of $\varphi(I)$. By computing Hilbert polynomial, one can check that $J = p(I)$. Therefore $\rA^{*}(\rN) \cong \QQ[b_{1},b_{2},b_{3},d_{2}]/J$. The presentation on the statement is obtained by eliminating $b_{3}$, by using the first relation $-b_{1}^{3}+4b_{1}d_{2}+3b_{3} = 0$. 
\end{proof}

\begin{remark}
Let $\cF$ and $\cE$ be two universal bundles on $\rN$ and $\varphi : \cF \to V^* \otimes \cE$ be the universal morphism. Then $b_{1}, b_{2}, b_{3}$ are the Chern classes of $\cE$ and $d_{1}, d_{2}$ are Chern classes of $\cF$ (\cite[(6.11)]{ES89}). 
\end{remark}


\section{Chow ring of $\rM^{+}$}\label{sec:ChowMp}

By using the blow-up description of the projective morphism $\rM^{+} \to \rQ \to \rN$ in \eqref{eqn:diagramSec3}, we are able to compute the Chow ring of $\rM^{+}$. 

\subsection{Chow ring of $\rQ$}\label{cohQ}

Recall that $\rQ$ is the projective bundle $\PP \cU$ over $\rN$ of a rank 12 locally free sheaf $\cU$, which is the cokernel of the injective morphism 
\[
	\cF \otimes \bH^{0}(\cO_{\PP^{2}}(1)) \to 
	\cE \otimes \bH^{0}(\cO_{\PP^{2}}(2)), 
\]
where $\cF$ and $\cE$ are two universal sheaves on $\rN$. 
Let $\rho := c_{1}(\cO_{\rQ}(1))$. 

\begin{proposition}
\[
\begin{split}
	\rA^{*}(\rQ) \cong\;& 
	\QQ[\rho, b_{1}, b_{2}, d_{2}]/
	\langle b_{1}^{2}d_{2} - 3d_{2}^{2}, 
	b_{1}^{2}b_{2} - b_{2}d_{2} - 3d_{2}^{2},
	b_{1}^{4}+3b_{2}^{2} - 9b_{2} d_{2} - 3d_{2}^{2},
	2b_{1}b_{2}d_{2} - 3b_{1}d_{2}^{2},\\
	& 3b_{1}b_{2}^{2} - 7b_{1}d_{1}^{2}, 
	\rho^{12} + 3\rho^{11}b_{1} + 
	3\rho^{10}(b_{1}^{2} + 2b_{2} - d_{2}) +
	\rho^{9}(-b_{1}^{3} + 12b_{1}b_{2} + 2b_{1}d_{2}) + \\
	& + 3\rho^{8}(9b_{2}^{2} - 16b_{2}d_{2} + 17d_{2}^{2}) + 
	28\rho^{7}b_{1}d_{2}^{2}+56\rho^{6}d_{2}^{3}\rangle
\end{split}
\]
\end{proposition}

\begin{proof}
Since $\rQ$ is a projective bundle $\PP(\cU)$ over $\rN$, by \cite[Example 8.3.4]{Ful98},
\[
	\rA^{*}(\rQ) = \rA^{*}(\rN)[\rho]/\langle
	\sum_{i=0}^{12}\rho^{i}c_{12-i}(\cU)\rangle.
\]
By \eqref{eqn:exactseqforQ},
\begin{equation}\label{eqn:ChernclassU}
\begin{split}
	c(\cU) =&\;
	(c(\cE))^{6}/(c(\cF))^{3} =
	(1+b_{1}+b_{2}+b_{3})^{6}/(1+b_{1}+d_{2})^{3}\\
	= &\; 1 + 3b_{1} + 3b_{1}^{2} + 6b_{2} - 3d_{2}
	-b_{1}^{3} + 12b_{1}b_{2} + 2b_{1}d_{2}
	+ 27b_{2}^{2} - 48b_{2}d_{2} + 51d_{2}^{2}
	+ 28b_{1}d_{2}^{2}+56d_{2}^{3}
\end{split}
\end{equation}
in $\rA^{*}(\rN)$. Therefore we obtain the presentation. 
\end{proof}

\subsection{A digression into universal plane curves}\label{ssec:univcurve}

Let $\cC_d$ be the total space of degree $d$ plane curves, which is the incidence variety 
\[
	\cC_{d} := \{(p, C)\;|\; p \in C\} \subset 
	\PP^{2} \times |\cO_{\PP^{2}}(d)|.
\]
Then $\cC_d$ is isomorphic to a projective bundle $\pi : \PP(\cE) \to \PP^{2}$ where $\cE$ is given by 
\[
	\ses{\cE}{\cO_{\PP^2}\otimes \rH^0(\cO_{\PP^2}(d))}{\cO_{\PP^2}(d)}. 
\]
Note that the rank of $\cE$ is $r_d:=\frac{d(d+3)}{2}$. By \cite[Example 8.3.4]{Ful98}, 
\begin{equation}\label{eqn:Chowuniversalcurve}
	\rA^*(\cC_d) \cong \QQ[k,\eta]/\langle k^3, 
	\eta^{r_d}-d\eta^{r_d-1}k+d^2\eta^{r_d-2}k^2 \rangle
\end{equation}
where $k$ is the hyperplane class of $\PP^{2}$ and $\eta=c_1(\cO_{\PP(\cE)}(1))$.

Also from the generalized Euler sequence (\cite[Example 3.2.11]{Ful98})
\[
	0 \to \cO_{\cC_{d}} \to \pi^{*}\cE \otimes \cO_{\cC_{d}}(1) \to 
	\cT_{\cC_{d}} \to \pi^{*}\cT_{\PP^{2}} \to 0, 
\]
we obtain the Chern character
\[
ch(\cC_{d}) = (r_{d}+1-e^{dk})e^{\eta}+3e^{k}-2.
\]
In particular, 
\begin{equation}\label{eqn:twoChernclassesCd}
	c_{1}(\cC_{d}) = (3-d)k+r_{d}\eta, \;
	c_{2}(\cC_{d}) = (d^{2}-3d+3)k^{2}+(3r_{d}-dr_{d}+d)k\eta
	+\frac{r_{d}(r_{d}-1)}{2}\eta^{2}.
\end{equation}

\subsection{Chow ring of $\rM^{+}$}
As we have seen in Proposition \ref{prop:mainprop}, $q : \rM^{+} \to \rQ$ is a regular blow-up of $\rQ$ and the blow-up center is isomorphic to $\mathrm{Fl}(V) \cong \cC_{1} \cong \PP\Omega_{\PP V}(1)$. By using the blow-up formula of Chow ring (\cite[Theorem 1 in Appendix]{Kee92}), we are able to obtain $\rA^{*}(\rM^{+})$. 

\begin{proposition}[\protect{\cite[Section 3.3.3]{DM11}}]\label{prop:moduliPVdual}
For the projection map $p : \rQ \to \rN$, the image $p(\mathrm{Fl}(V))$ parametrizes the extension classes $0 \to \cO_{\ell}(-1) \to G_{\ell} \to \cO_{\PP V}(1) \to 0$.
\end{proposition}
Because the extension determined by a choice of a line $\ell \subset \PP V$, $p(\mathrm{Fl}(V)) \cong \PP V^{*}$. 

Consider the following fiber product diagram:
\begin{equation}\label{eqn:squarediagram}
	\xymatrix{\mathrm{Fl}(V) \ar[r]^{i} \ar[d]_{p'} & \rQ \ar[d]^{p}\\
	\PP V^{*} \ar[r]^{j} & \rN}
\end{equation}

\begin{lemma}\label{lem:classesonN}
\begin{enumerate}
\item Let $\cF$ (resp. $\cE$) be the universal rank 2 (resp. rank 3) bundle on $\rN$. Then $j^{*}(\cF) \cong \Omega_{\PP V^{*}}$ and $j^{*}(\cE) \cong \wedge^{2}V^{*} \otimes \cO_{\PP V^{*}}(-1) \cong \cO_{\PP V^{*}}(-1)^{3}$.
\item Let $\rA^{*}(\PP V^{*}) = \QQ[h]/\langle h^{3}\rangle$ where $h$ is the hyperplane class. Then $j^{*} : \rA^{*}(\rN) \to \rA^{*}(\PP V^{*})$ is given by
\[
	j^{*}(b_{1}) = -3h, \; j^{*}(b_{2}) = 3h^{2}, \; j^{*}(b_{3}) = 0, \;
	j^{*}(d_{2}) = 3h^{2}.
\]
\item The class $[\PP V^{*}]$ in $\rA^{*}(\rN)$ is $-3b_{1}^{2}b_{2} + 5b_{1}^{2}d_{2}$.
\end{enumerate}
\end{lemma}

\begin{proof}
Fix $[\ell] \in \PP V^{*}$, which is determined by $z \in V^{*}$. Then on the fiber corresponding to $[\ell] \in \PP V^{*} \subset \rN$, the restriction of the universal map $\phi : \cF_{\ell} \to \cE_{\ell}\otimes V^{*}$ is given by $\cO_{\PP V}(-1) \otimes D \to \cO_{\PP V} \otimes \wedge^{2} V \cong \cO_{\PP V}\otimes V^{*}$, where $D$ is the plane in $V$ corresponding to $\ell$ (\cite[Section 3.3.3]{DM11}).

A twisted Euler sequence
\[
	0 \to \cO_{\PP V^{*}}(-1) \to \cO_{\PP V^{*}}\otimes V^{*} \to
	\cT_{\PP V^{*}}(-1) \to 0,
\]
is restricted to $0 \to \langle z \rangle \to V^{*} \to V^{*}/\langle z \rangle \to 0$ at $\ell$. So the dual sequence
\[
	0 \to \Omega_{\PP V^{*}}(1) \to \cO_{\PP V^{*}}\otimes V \to
	\cO_{\PP V^{*}}(1) \to 0
\]
is restricted to $0 \to D \to V \to V/D \to 0$. Thus we can see that the pull-back of the universal bundles and maps are
\[
	\cO_{\PP V}(-1) \otimes \Omega_{\PP V^{*}}(1) \otimes L \to
	\cO_{\PP V} \otimes \cO_{\PP V^{*}}\otimes \wedge^{2}V^{*} \otimes L
\]
for some line bundle $L \in \mathrm{Pic}(\PP V^{*})$. From $c_{1}(\cF) = c_{1}(\cE)$, it is straightforward to see that $L \cong \cO_{\PP V^{*}}(-1)$. This proves (1). Item (2) is an immediate consequence of Item (1).

By the result in Section \ref{sec:ChowofN}, we can describe any codimension 4 Chow class in $\rN$ as a linear combination $x b_{1}^{4} + y b_{1}^{2}b_{2} + z b_{1}^{2}d_{2}$ for $x, y, z \in \QQ$. By taking pull-backs of $b_{1}^{2}$, $b_{2}$ and $d_{2}$ and evaluating the intersection pairing on $\rA^{*}(\rN)$ (for instance, see \cite[p.23, Table 2]{Tjo98}), we obtain three linear equations
\begin{eqnarray*}
	57x + 27y + 18z &=& 9\\
	27x + 14y + 9z &=& 3\\
	18x + 9y + 6z &=& 3.
\end{eqnarray*}
The unique solution is $(x, y, z) = (0, -3, 5)$.
\end{proof}

\begin{corollary}
The normal bundle $\cN_{\PP V^{*}/\rN}$ has the total Chern class 
\[
	c(\cN_{\PP V^{*}/\rN}) = 1 + 6h + 21h^{2}.
\]
\end{corollary}

\begin{proof}
From \eqref{eqn:ChernclassN} and Lemma \ref{lem:classesonN},
\[
	c(\cT_{\rN}|_{\PP V^{*}}) = 1 + 9h + 42 h^{2}.
\]
Then
\[
	c(\cN_{\PP V^{*}/\rN})
	= \frac{c(\cT_{\rN}|_{\PP V^{*}})}{c(\cT_{\PP V^{*}})}
	= \frac{1+9h + 42h^{2}}{1+3h + 3h^{2}}
	= 1 + 6h + 21h^{2}.
\]
\end{proof}

From Section \ref{ssec:univcurve}, 
\[
	\rA^{*}(\mathrm{Fl}(V)) = \rA^{*}(\cC_{1}) \cong 
	\QQ[h, k]/\langle k^{3}, h^{2} - hk + k^{2}\rangle
\]
where $k$ is the pull-back of the hyperplane class on $\PP V$. One may check that $h=\cO_{\PP(\cE)}(1)$ is isomorphic to the pull-back of the hyperplane class on $\PP V^{*}$. On the other hand, since $\mathrm{Fl}(V) \cong \PP \cT_{\PP V^{*}}(-1)$, from \cite[Example 8.3.4]{Ful98}, 
\[
	\rA^{*}(\mathrm{Fl}(V)) = 
	\rA^{*}(\PP \cT_{\PP V^{*}}(-1), \QQ) \cong
	\QQ[h, \xi]/\langle h^{3}, h^{2} + h\xi + \xi^{2}\rangle
\]
where $\xi = c_{1}(\cO_{\PP \cT_{\PP V^{*}}(-1)}(1))$. Then $k = h + \xi$. We leave checking it to the reader. 

\begin{lemma}\label{lem:restrictiontoFlV}
The restriction map $i^{*} : \rA^{*}(\rQ) \to \rA^{*}(\mathrm{Fl}(V))$ is given by
\[
	i^{*}(b_{1}) = -3h, \; i^{*}(b_{2}) = 3h^{2}, \; i^{*}(b_{3}) = 0,\;
	i^{*}(c_{2}) = 3h^{2}, \; i^{*}(\rho) = \xi.
\]
\end{lemma}

\begin{proof}
It suffices to check the last equality. Recall that there is a restricted exact sequence
\[
	\xymatrix{ 0 \ar[r] & j^{*}\cF \otimes \rH^{0}(\cO_{\PP V}(1))
	\ar[r] \ar@{=}[d] & j^{*}\cE \otimes \rH^{0}(\cO_{\PP V}(2))
	\ar[r] \ar@{=}[d] & j^{*}\cU \ar[r] & 0.\\
	& \Omega_{\PP V^{*}} \otimes \rH^{0}(\cO_{\PP V}(1))
	& \cO_{\PP V^{*}}(-1)^{3} \otimes \rH^{0}(\cO_{\PP V}(2))}
\]
Then $c(j^{*}\cU) = 1 - 9h + 36h^{2}$.

From Proposition \ref{prop:moduliPVdual}, if we take the pull-back $j^{*}$, we have another exact sequence
\[
	0 \to \pi_{1 *}(\cO_{U}\otimes \pi_{2}^{*}\cO_{\PP V}(1)) 
	\otimes A \to j^{*}\cU \to
	\rH^{0}(\cO_{\PP V}(3)) \otimes B\to 0.
\]
where $U \subset \PP V^{*} \times \PP V$ is the universal line over $\PP V^{*}$, and $A, B \in \mathrm{Pic}(\PP V^{*})$. Two maps $\pi_{1} : U \to \PP V^{*}$ and $\pi_{2} : U \to \PP V$ are projections. From the exact sequence
\[
	0 \to \cO_{\PP V^{*} \times \PP V}(-1, -1) \to
	\cO_{\PP V^{*} \times \PP V} \to \cO_{U} \to 0,
\]
we have
\[
	0 \to \cO_{\PP V^{*}}(-1) \to \cO_{\PP V^{*}} \otimes V^{*}
	\to \pi_{1 *}(\cO_{U}\otimes \pi_{2}^{*}\cO_{\PP V}(1)) \to 0
\]
by tensoring $\pi_{2}^{*}(\cO_{\PP V}(1))$ and taking push-forward $\pi_{1 *} : \PP V^{*} \times \PP V \to \PP V^{*}$. Thus $\pi_{1 *}(\cO_{U}\otimes \pi_{2}^{*}\cO_{\PP V}(1)) \cong \cT_{\PP V^{*}}(-1)$. Therefore we have an exact sequence
\begin{equation}\label{eqn:restricteduniversalbundle}
	0 \to \cT_{\PP V^{*}}(-1) \otimes A \to j^{*}\cU \to B^{10} \to 0.
\end{equation}
From $c(\cT_{\PP V^{*}}(-1) \otimes A)c(B^{10}) = c(j^{*}\cU) = 1 - 9h + 36h^{2}$, we obtain that $A \cong \cO_{\PP V^{*}}$ and $B \cong \cO_{\PP V^{*}}(-1)$. Thus there is a natural inclusion $\mathrm{Fl}(V) = \PP \cT_{\PP V^{*}}(-1) \hookrightarrow \PP j^{*}\cU$, and $j^{*}\cO_{\rQ}(-1) = \cO_{\PP \cT_{\PP V^{*}}(-1)}(-1)$. Therefore $i^{*}\rho = \xi$.
\end{proof}

Let $P := p^{-1}(\PP V^{*}) \cong \PP j^{*}\cU \subset \rQ$. Then we have a normal sequence
\begin{equation}\label{eqn:normalsequence}
	0 \to \cN_{\PP \cT_{\PP V^{*}}(-1)/P} \to
	\cN_{\PP \cT_{\PP V^{*}}(-1)/\rQ} \to
	\cN_{P/Q_{4}} =p^{*}\cN_{\PP V^{*}/N} \to 0.
\end{equation}

\begin{lemma}\label{lem:normaltoP}
The normal bundle $\cN_{\PP \cT_{\PP V^{*}}/P}$ is isomorphic to $(p^{*}\cO_{\PP V^{*}}(-1) \otimes \cO_{\PP \cT_{\PP V^{*}}(-1)}(1))^{10}$.
\end{lemma}

\begin{proof}
From \eqref{eqn:restricteduniversalbundle} and Euler sequence, we can construct a commutative diagram
\[
	\xymatrix{&& 0 \ar[d] & 0 \ar[d]\\
	0 \ar[r] & \cO_{\PP \cT_{\PP V^{*}}(-1)} \ar@{=}[d] \ar[r] &
	p^{*}\cT_{\PP V^{*}}(-1) \otimes \cO_{\PP \cT_{\PP V^{*}}(-1)}(1)
	 \ar[r] \ar[d] &
	\cT_{\PP \cT_{\PP V^{*}}(-1)} \ar[r] \ar[d] & 0\\
	0 \ar[r] & \cO_{\PP \cT_{\PP V^{*}}(-1)} \ar[r] &
	j^{*}\cU \otimes \cO_{\PP \cT_{\PP V^{*}}(-1)}(1) \ar[d] \ar[r] &
	 \cT_{P} \ar[r] \ar[d] & 0\\
	&& p^{*}\cO_{\PP V^{*}}(-1)^{10} \otimes
	\cO_{\PP \cT_{\PP V^{*}}(-1)}(1) \ar[r] \ar[d] &
	\cN_{\PP \cT_{\PP V^{*}}(-1)/P} \ar[d]\\
	&& 0 & 0.}
\]
The bottom map $p^{*}\cO_{\PP V^{*}}(-1)^{10} \otimes \cO_{\PP \cT_{\PP V^{*}}(-1)}(1) \to \cN_{\PP \cT_{\PP V^{*}}(-1)/P}$ is locally isomorphic, thus 
\[
	\cN_{\PP \cT_{\PP V^{*}}(-1)/P} \cong 
	p^{*}\cO_{\PP V^{*}}(-1)^{10} \otimes 
	\cO_{\PP \cT_{\PP V^{*}}(-1)}(1) \cong 
	(p^{*}\cO_{\PP V^{*}}(-1) \otimes 
	\cO_{\PP \cT_{\PP V^{*}}(-1)}(1))^{10}.
\]
\end{proof}

\begin{corollary}\label{cor:normalbundleFlV}
\[
	c(\cN_{\mathrm{Fl}(V)/\rQ})
	= 1 + 10\xi -4h -75\xi h - 39 h^{2} + 120 \xi h^{2}
\]
\end{corollary}
\begin{proof}
By Lemma \ref{lem:normaltoP} and \eqref{eqn:normalsequence}, 
\begin{eqnarray*}
	c(\cN_{\mathrm{Fl}(V)/\rQ}) &=&
	c(\cN_{\mathrm{Fl}(V)/P})c(\cN_{\PP V^{*}/\rN})
	= (1+\xi-h)^{10}(1+6h+21h^{2})\\
	&=& 1 + 10\xi -4h -75\xi h - 39 h^{2} + 120 \xi h^{2}.
\end{eqnarray*}
\end{proof}

\begin{corollary}
The class of $[\mathrm{Fl}(V)]$ in $\rA^{*}(\rQ)$ is $(\rho+\frac{1}{3}b_{1})^{10}(-3b_{1}^{2}b_{2}+5b_{1}^{2}c_{2})$.
\end{corollary}

\begin{proof}
In $\rQ$, the class of $[P]$ is $-3b_{1}^{2}b_{2}+5b_{1}^{2}c_{2}$ by Lemma \ref{lem:classesonN}. By Lemma \ref{lem:normaltoP}, in $P$, $[\mathrm{Fl}(V)]$ is the zero section of $(\pi^{*}\cO_{\PP V^{*}}(-1) \otimes \cO_{\PP \cT_{\PP V^{*}}(-1)}(1))^{10}$. This implies the result.
\end{proof}

Now we are ready to apply the blow-up formula.
\begin{proposition}
\[
\begin{split}
	\rA^{*}(\rM^{+}) =\;& \QQ[\tau, \rho, b_{1}, b_{2}, d_{2}]/
	\langle b_1^2 d_2-3 d_2^2,b_1^2 b_2-b_2 d_2-3 d_2^2,
	b_1^4+3b_2^2-9b_2 d_2-3d_2^2,2b_1 b_2 d_2-3 b_1 d_2^2,\\
	& 3b_1 b_2^2-7b_1 d_2^2,
	\rho^{12} + 3\rho^{11}b_{1}
	+ 3\rho^{10}(b_{1}^{2} + 2b_{2} - d_{2})
	+\rho^{9}(-b_{1}^{3} + 12b_{1}b_{2} + 2b_{1}d_{2})\\
	&+ 3\rho^{8}(9b_{2}^{2} - 16b_{2}d_{2} + 17d_{2}^{2})
	+ 28\rho^{7}b_{1}d_{2}^{2}+56\rho^{6}d_{2}^{3},\\
	& \tau(b_{2}-d_{2}), \tau(b_{1}^{2}-3d_{2}), 
	\tau(d_{2}-b_{1}\rho+3\rho^{2}), 
	\tau b_{1}b_{2}, \tau d_{2}^{2},\\
	& \tau^{14} + (10\rho + \frac{4}{3}b_{1})\tau^{13}
	+ (25\rho b_{1} - 13b_{2})\tau^{12} + 40 \rho b_{2}\tau^{11}+
	(\rho+\frac{1}{3}b_{1})^{10}(-3b_{1}^{2}b_{2}+5b_{1}^{2}d_{2})
	\rangle
\end{split}
\]
\end{proposition}

\begin{proof}
By the blow-up formula (\cite[Theorem 1 in Appendix]{Kee92}),
\[
	\rA^{*}(\rM^{+}) =\rA^{*}(\rQ)[\tau]/\langle
	\tau \ker i^{*}, \tau^{14} + c_{1}(\cN)\tau^{13} +c_{2}(\cN)\tau^{12} +
	c_{3}(\cN)\tau^{11}+ [\mathrm{Fl}(V)]
	\rangle
\]
where $\cN = \cN_{\mathrm{Fl}(V)/\rQ}$. 

It is straightforward to check that $J := \langle b_{2}-d_{2},b_{1}^{2}-3d_{2},d_{2}-b_{1}\rho+3\rho^{2},b_{1}b_{2},d_{2}^{2}\rangle \subset \ker i^{*}$. Furthermore, by Lemma \ref{lem:restrictiontoFlV}, $i^{*} : \rA^{*}(\rQ) \to \rA^{*}(\mathrm{Fl}(V))$ is surjective, thus $\rA^{*}(\rQ)/\ker i^{*} \cong \rA^{*}(\mathrm{Fl}(V))$ and their Hilbert polynomials are same. We can show that $P_{t}((\rA^{*}(\rQ)/J) = 1 + 2t + 2t^{2} + t^{3} = P_{t}(\mathrm{Fl}(V))$. Thus $\ker i^{*} = J$. \end{proof}

\begin{remark}
Instead of using $\rM^{+} \to \rQ \to \rN$, by applying the blow-up/down computation to the flip $\rM^\infty \dashrightarrow \rM^+$, one may obtain the Chow ring of $\rM^+$ from that of $\rM^\infty$. But this computation seems to be more complicated than that in this paper.
\end{remark}

\section{Chow ring of $\rM$}\label{sec:ChowM}

Recall that the exceptional divisor $E$ of the blow-up $q : \rM^{+} \to \rQ$ has two fibration structures:
\[
	\xymatrix{&E=\mathrm{Fl}(V) \times_{\PP V}\cC_{4}
	\ar[ld]_{t} \ar[rd]^{s}\\
	\PP V^{*}\times \PP V \supset \mathrm{Fl}(V) \ar[rd]^{u}&&
	\cC_{4} \subset \PP V \times |\cO_{\PP^{2}}(4)| \ar[ld]_{v}\\
	&\PP V.}
\]
By contracting $E$ to $\cC_{4}$, we obtain a blow-down map $r : \rM^{+} \to \rM$. In this section, by using this contraction, we compute the Chow ring of $\rM$. 

Consider the following Cartesian diagram:
\begin{equation}\label{eqn:blowupdiagramforM}
	\xymatrix{E \ar[r]^{i} \ar[d]_{r'} & \rM^{+} \ar[d]^{r}\\ 
	\cC_{4} \ar[r]^{j} & \rM}
\end{equation}
We have a morphism of Chow groups $r_{*} : \rA_{*}(\rM^{+}) \to \rA_{*}(\rM)$ and a morphism of Chow rings (so called the refined Gysin map) $r^{*}: \rA^{*}(\rM) \to \rA^{*}(\rM^{+})$ such that $r_{*}r^{*} = \mathrm{id}$ (\cite[Proposition 6.7.(b)]{Ful98}). In particular, $r^{*}$ is an injective morphism and $\rA^{*}(\rM)$ can be regarded as a subring of $\rA^{*}(\rM^{+})$. To find the image $r^{*}$, we will apply the following result. 

\begin{proposition}\cite[Lemma 10.5]{AV90}\label{prop:pullbackformula}
Let $f : \widetilde{X} \to X$ be the blow-up of $n$-dimensional smooth projective variety $X$ along a smooth subvariety $Y$. Let $j : \widetilde{Y} \to \widetilde{X}$ be the exceptional divisor and $g : \widetilde{Y} \to Y$ be the projection. Assume that the numerical equivalence is equivalent to the rational equivalence. Then
\[
	f^{*}(\rA_{i}(X)) = (j_{*}(\ker g_{*})_{n-i})^{\perp}.
\]
\end{proposition}

Besides, from the blow-up description of $\rM$, the Poincar\'e polynomial of $\rM$ can be obtained. 
\begin{proposition}[\protect{\cite[Corollary 5.2]{CC12}}]\label{prop:ppolyM}
\begin{multline*}
P_{t}(\rM) = 
1+2t+6t^2+10t^3+14t^4+15t^{5}+16t^{6}+16t^{7}+16t^{8} \\
+16t^{9}+16t^{10}+16t^{11}+15t^{12}+14t^{13}+10t^{14}+6t^{15}+2t^{16}+t^{17}.
\end{multline*}
\end{proposition}
Thus it is sufficient to find enough generators in the image, which generate a subring of given Poincar\'e polynomial. 

\begin{lemma}\label{lem:ChowE}
\begin{enumerate}
\item The Chow ring $\rA^{*}(E)$ is isomorphic to $\QQ[k, h, \eta]/\langle {k}^{3}, \eta^{14} - 4k\eta^{13} + 16k^{2}\eta^{12}, h^{2} - kh + k^{2}\rangle$ where $k$ (resp. $h$) is the pull-back of the hyperplane class in $\PP V$ (resp. $\PP V^{*}$), and $\eta$ is the pull-back of $c_{1}(\cO_{\PP(\cE)}(1))$ (for the notation, see Section \ref{ssec:univcurve}). 
\item The kernel of $r'_{*} : \rA_{*}(E) \to \rA_{*}(\cC_{4})$ is isomorphic to $\QQ[k, \eta]/\langle {k}^{3},\eta^{14} - 4k\eta^{13} + 16k^{2}\eta^{2}\rangle$.
\end{enumerate}
\end{lemma}

\begin{proof}
By Section \ref{ssec:univcurve}, $\rA^{*}(\cC_{4}) \cong \QQ[k, \eta]/\langle k^{3}, \eta^{14} - 4k\eta^{13} + 16k^{2}\eta^{2}\rangle$. On $\PP V$, we have a standard exact sequence 
\[
	0 \to K \to \cO_{\PP V}\otimes \rH^{0}(\cO_{\PP V}(1)) \to 
	\cO_{\PP V}(1) \to 0
\]
and $E = \PP v^{*}(K)$. Thus by \cite[Example 8.3.4]{Ful98} again, 
\[
	\rA^{*}(E) \cong \QQ[k, h, \eta]/\langle k^{3},
	\eta^{14} - 4k\eta^{13} + 16k^{2}\eta^{12},
	h^{2} - kh + k^{2}\rangle.
\]
Note that the subring generated by $1, k, \eta$ is in $\ker r'_{*}$, because ${r'}: E \to \cC_{4}$ is a positive dimensional smooth fibration. By projection formula, $r'_{*}(h\alpha) = \alpha$ for $\alpha \in \rA^{*}(\cC_{4})$. Therefore 
\[
	\ker r'_{*}\cong \QQ[k, \eta]/\langle k^{3},
	\eta^{14} - 4k\eta^{13} + 16k^{2}\eta^{2}\rangle.
\]
\end{proof}

\begin{remark}
From another fibration structure $u \circ t : E \to \mathrm{Fl}(V) \to \PP V$, by using Corollary \ref{cor:normalbundleFlV} (note that $\xi = k-h$), we can obtain another presentation
\[
	\rA^{*}(E) \cong \QQ[k, h, \tau]/\langle k^{3}, k^{2}-kh + h^{2},
	\tau^{14} + (10k-14h)\tau^{13} + (-75kh+36h^{2})\tau^{12}
	+ 120kh^{2}\tau^{11}\rangle
\]
where $\tau = c_{1}(\cO_{\PP\cN_{\mathrm{Fl}(V)/\rQ}}(1))$. It is straightforward to check that $\eta$ in Lemma \ref{lem:ChowE} is $\tau + k - h$. 
\end{remark}

Now we are able to compute $\rA^{*}(\rM)$. 
\begin{theorem}\label{thm:ChowM}
\[
\begin{split}
\rA^{*}(\rM) \cong \;& \QQ[\alpha, \beta, x, y, z]/\langle
xz-yz,\beta^2 z-3yz-9z^2,3\alpha^2 z-\alpha\beta z+yz,\beta^2 y-3y^2-9yz,\\
&\beta^2 x-xy-3y^2-3\alpha \beta z-9yz+9z^2, \beta^4+3x^2-9xy-3y^2-54yz-81z^2,\\
&\beta yz+9\alpha z^2-3\beta z^2, 2\beta xy-3\beta y^2-9\alpha yz-27\alpha z^2+9\beta z^2, 3\beta x^2-7\beta y^2-36\alpha yz\\
&-108\alpha z^2+36\beta z^2,
\alpha^{12}+3\alpha^{11}\beta+3\alpha^{10}(\beta^2+2x-y)
+\alpha^9(-\beta^3+12\beta x+2\beta y)\\
&+3\alpha^{8}(9x^{2}-16xy+27y^{2})+28\alpha^7\beta y^2
+56\alpha^6 y^3+201\alpha\beta z^5-19yz^5-613z^6,\\
& 6\alpha^{10}xy-12\alpha^{10}y^2-10\alpha^9 \beta y^2-45\alpha^8 y^3-104\alpha\beta z^6+2yz^6+310z^7\rangle,
\end{split}
\]
where $\alpha, \beta$ are of degree 1 and $x, y, z$ are of degree 2. The inclusion map $r^{*} : \rA^{*}(\rM)\to \rA^{*}(\rM^{+})$ is given by $r^{*}(\alpha) = \rho + \tau$, $r^{*}(\beta) = b_{1} + 3\tau$, $r^{*}(x) = b_{2}-(3\rho-b_{1})\tau$, $r^{*}(y) = d_{2}-(3\rho-b_{1})\tau$, and $r^{*}(z) = \tau(\tau+\rho+b_{1}/3)$. 
\end{theorem}

\begin{proof}
By Proposition \ref{prop:pullbackformula}, $\gamma \in \rA^{*}(\rM^{+})$ is in $r^{*}(\rA^{*}(\rM))$ if and only if for every $\delta \in \ker s_{*}$, $i_{*}(\delta)\cdot \gamma = 0$. By the projection formula, $i_{*}(\delta \cdot i^{*}(\gamma)) = 0$. Thus if $\delta\cdot i^{*}(\gamma) = 0$, then $\gamma \in r^{*}(\rA^{*}(\rM))$. 

From the description of $\ker r'_{*}$ in Lemma \ref{lem:ChowE}, $(\ker r'_{*})_{2}$ is generated by $k^{2}\eta^{12}$ and $k\eta^{13}$, and $(\ker r'_{*})_{1} = \langle k^{2}\eta^{13}\rangle$. On the other hand, the pull-back $i^{*} : \rA^{*}(\mathrm{M}^{+}) \to \rA^{*}(E)$ is given by $i^{*}(\rho) = \xi = k - h$, $i^{*}(\tau) = \eta - k + h$, $i^{*}(b_{1}) = -3h$, $i^{*}(b_{2}) = 3h^{2}$, and $i^{*}(c_{2}) = 3h^{2}$ by Lemma \ref{lem:restrictiontoFlV}.

For $\rho \in \rA^{*}(\rM^{+})$, $r^{*}r_{*}(\rho) = \rho + n\tau$ for some $n \in \ZZ$. Because $i^{*}(\rho + n \tau) = k - h + n(\eta - k + h) = n\eta - (n-1)(k - h)$. Then $k^{2}\eta^{13}(n\eta - (n-1)(k - h)) = (n-1)hk^{2}\eta^{13} = 0$ only if $n = 1$. So $\rho + \tau \in r^{*}(\rA^{*}(\rM))$. By a similar computation, $b_{1}+3\tau \in r^{*}(\rA^{*}(\rM))$.

Now $r^{*}r_{*}(b_{2}) = b_{2} + a\tau^{2} + b\tau\rho + c\tau b_{1}$ for some $a, b, c \in \QQ$. By intersecting with $k^{2}\eta^{12}$ and $k\eta^{13}$, we can obtain two linear equations:
\begin{eqnarray*}
	3 + 7a - 3b - 12c &=&0,\\
	2a - b - 3c &=&0.
\end{eqnarray*}
By solving this system, we obtain $b_{2} - (3\rho - b_{1})\tau, b_{2}+(3\tau+2b_{1})\tau \in r^{*}(\rA(\rM))$. Then $(b_{2} + (3\rho - b_{1})\tau-b_{2}+(3\tau+2b_{1})\tau )/3= \tau(\tau+\rho+b_{1}/3) \in r^{*}(\rA^{*}(\rM))$, too. By a similar computation, $d_{2} - (3\rho-b_{1})\tau \in r^{*}(\rA^{*}(\rM))$. 

In summary, a subring $R$ generated by $\alpha:= \rho + \tau, \beta:= b_{1} + 3\tau, x:= b_{2}-(3\rho-b_{1})\tau, y:=d_{2} - (3\rho-b_{1})\tau, z:=\tau(\tau + \rho+b_{1}/3)$ is a subring of $r^{*}(\rA^{*}(\rM))$. By using a computer algebra system, one can compute the presentation of $R$ on the statement, and check the fact that Hilbert series of $R$ and $\rA^{*}(\rM)$ are same. Therefore $R = r^{*}(\rA^{*}(\rM)) \cong \rA^{*}(\rM)$. 
\end{proof}

\section{Chern classes and effective cycles}\label{sec:Cherneff}

The aim of this section is twofold. First of all, by computing Chern classes and some effective classes, we want to obtain geometric description of the generators of $\rA^{*}(\rM)$. Secondly, we will use Chern classes of $\rM$ to compute Euler characteristics of line bundles on $\rM$ in the next section. This result can be used to provide some numerical data on the strange duality conjecture of Le Potier. 

\subsection{Chern classes}

By \cite[Equation (3.4)]{Tjo98}, there is an exact sequence 
\[
	0 \to \cO_{\rN} \to \cEnd(\cE) \oplus \cEnd(\cF) \to 
	\cHom(\cF, H \otimes \cE) \to \cT_{\rN} \to 0,
\]
where $\cE$ and $\cF$ are two universal bundles on $\rN$. Thus we can compute the total Chern class of the tangent bundle of $\rN$, which is:
\begin{equation}\label{eqn:ChernclassN}
	c(\cT_{\rN}) = 
	1-3b_1+3b_1^2+5d_2-3b_1^3-4b_1d_2-9b_2^2+27b_2 d_2+4d_2^2
	-\frac{17}{2}b_1d_2^2+\frac{13}{2}d_2^3.
\end{equation}

On $\rQ \cong \PP \cU$, there is an exact sequence (\cite[Example 3.2.11]{Ful98})
\[
	0 \to \cO_{\rQ} \to p^{*}\cU \otimes \cO_{\rQ}(1) 
	\to \cT_{\rQ} \to p^{*}\cT_{\rN} \to 0.
\]
Apply \cite[Example 3.2.2]{Ful98} and \eqref{eqn:ChernclassU}, then we have:
\[
	c(\cT_{\rQ}) = c(p^{*}\cU \otimes \cO_{\rQ}(1))
	c(p^{*}\cT_{\rN}) = 
	\left(\sum_{i\ge 0}\sum_{j=0}^{i}{12-j \choose i-j}
	c_{j}(\cU)\rho^{i-j}\right)p^{*}c(\cT_{\rN}).
\]
For instance, 
\[
	c_{1}(\cT_{\rQ}) = 12\rho, \quad 
	c_{2}(\cT_{\rQ}) = 66\rho^{2} -3\rho b_{1} 
	- 3b_{1}^{2}+6b_{2}+2d_{2}.
\]
\begin{theorem}[\protect{\cite[Theorem 15.4, Example 15.4.1]{Ful98}}]\label{thm:blowupChernclass}
Let $f : \widetilde{Y} \to Y$ be the blow-up of a nonsingular variety $Y$ along a codimension $d$ nonsingular subvariety $X$. Let $j : E \to \widetilde{Y}$ be the inclusion of the exceptional divisor, $g : E \to X$ be the projection, $N = \cN_{X/Y}$, and $\zeta = c_{1}(\cO_{E}(1))$. Then 
\[
	c(\widetilde{Y}) - f^{*}c(Y) = j_{*}(g^{*}c(X) \cdot \alpha)
\]
where 
\[
	\alpha = \sum_{j=0}^{d}\sum_{k=0}^{d-j}
	\left({d-j \choose k}-{d-j \choose k+1}\right)
	\zeta^{k}g^{*}c_{j}(N).
\]
\end{theorem}

By using this theorem, we are able to compute $c(\cT_{\rM^{+}})$. For example, 
\[
	c_{1}(\cT_{\rM^{+}}) = 12\rho + 13\tau, \quad
	c_{2}(\cT_{\rM^{+}}) = 77\tau^{2}+146\tau\rho + 66\rho^{2}
	- \frac{4}{3}\tau b_{1} - 3\rho b_{1} - 3b_{1}^{2}+6b_{2}+2d_{2}.
\]

One obstacle to apply this theorem to the blow-down $r : \rM^{+} \to \rM$ is that a priori we do not know Chern classes of normal bundle. But without this information, we can compute the first Chern class 
\[
	r^{*}c_{1}(\cT_{\rM}) = c_{1}(\cT_{\rM^{+}}) -\tau 
	= 12(\tau + \rho).
\]
Since $\cC_{4} \hookrightarrow \rM$ is codimension two, Theorem \ref{thm:blowupChernclass} is specialized to (\cite[Example 15.4.3]{Ful98})
\begin{eqnarray*}
	r^{*}c_{2}(\cT_{\rM}) &=& c_{2}(\cT_{\rM^{+}}) 
	+ i_{*}r^{*}c_{1}(\cT_{\cC_{4}}) + [E]^{2} 
	= c_{2}(\cT_{\cM^{+}}) - r^{*}j_{*}[\cC_{4}] 
	+ r^{*}c_{1}(\cT_{\rM})[E]\\
	&=& 77\tau^{2}+146\tau \rho + 66\rho^{2}-\frac{4}{3}\tau b_{1}
	- 3\rho b_{1} - 3b_{1}^{2}+6b_{2}+2d_{2} - r^{*}[\cC_{4}]
	-\tau 12(\tau + \rho).
\end{eqnarray*}
By Proposition \ref{prop:BNlocus} which we will prove later (Note that in the proof of Proposition \ref{prop:BNlocus}, we use $r^{*}c_{1}(\cT_{\rM})$ only.), $r^{*}[\cC_{4}] = \tau(\tau + \rho + \frac{1}{3}b_{1})$. Therefore we obtain the following result. 
\begin{lemma}\label{lem:twoChernclasses}
\[
\begin{split}
	r^{*}c_{1}(\cT_{\rM}) =&\; 12(\tau + \rho) = 12\alpha,\\
	r^{*}c_{2}(\cT_{\rM}) =&\; 64\tau^{2}+133\tau\rho+66\rho^{2}
	-\frac{5}{3}\tau b_{1}-3\rho b_{1}-3b_{1}^{2}+6b_{2}+2d_{2}\\
	=&\; 66\alpha^{2}-3\alpha \beta - 3\beta^{2}+6x+2y+34z.
\end{split}
\]
\end{lemma}
\begin{lemma}
For $j : \cC_{4} \to \rM$, 
\[
	j^{*}(\rho+\tau) =  \eta, j^{*}(b_{1}+3\tau) = 3(\eta-k), 
	j^{*}(b_{2}-(3\rho-b_{1})\tau) = j^{*}(d_{2}-3(\rho-b_{1})\tau)
	= -3k\eta, 
\]
\[
	j^{*}(\tau(\tau+\rho+\frac{1}{3}b_{1})) = 
	\eta^{2}-k\eta+k^{2}.
\]
In particular, $j^{*}: \rA^{*}(\rM) \to \rA^{*}(\cC_{4})$ is surjective. 
\end{lemma}
\begin{proof}
It is a straightforward computation using Lemma \ref{lem:restrictiontoFlV} and some relations in $\rA^{*}(E)$, including $k = h+\xi$, $\eta = \tau+k-h$, and $k^{2}-k\xi+\xi^{2} = h^{2}+h\xi+\xi^{2}=0$. 
\end{proof}
The following lemma is an immediate consequence of \eqref{eqn:twoChernclassesCd} and above two lemmas. 
\begin{lemma}
Let $\cN = \cN_{\cC_{4}/\rM}$ be the normal bundle. Then 
\[
	c(\cN) = 1 +(k-2\eta) + (k^{2}-k\eta + \eta^{2}).
\]
\end{lemma}

By applying Theorem \ref{thm:blowupChernclass} again, we obtain $r^{*}c(\cT_{\rM})$. Since $r^{*}$ is an injective morphism, theoretically the computation of an expression of the Chern class in terms of $\alpha, \beta, x, y, z$ is a linear algebra problem, but the actual computation is very cumbersome. David Swinarski kindly made a Macaulay2 code to find such an expression.

\begin{proposition}\label{prop:Chernclass}
The total Chern class is
\[
\begin{split} 
&\;c(\cT_{\rM}) = 1+12\alpha+66 \alpha^2-3 \alpha \beta-3 \beta^2+6 x+2 y+34 z\\
&\;+220 \alpha^3-33 \alpha^2 \beta-33 \alpha \beta^2-4 \beta^3+60 \alpha x-6 \beta x+30 \alpha y+22 \beta y+414 \alpha z+22 \beta z\\
&\;+495 \alpha^4-165 \alpha^3 \beta-162 \alpha^2 \beta^2-36 \alpha \beta^3+270 \alpha^2 x-72 \alpha \beta x+9 x^2+195 \alpha^2 y+225 \alpha \beta y+18 x y\\
&\;-41 y^2+877 \alpha \beta z-774 y z+163 z^2\\
&\;+792 \alpha^5-495 \alpha^4 \beta-465 \alpha^3 \beta^2-144 \alpha^2 \beta^3+720 \alpha^3 x-378 \alpha^2 \beta x+54 \alpha x^2+740 \alpha^3 y+1038 \alpha^2 \beta y\\
&\;+258 \alpha x y-450 \alpha y^2+\beta y^2-123 \alpha y z+16667 \alpha z^2-(7537/3) \beta z^2\\
&\;+924 \alpha^6-990 \alpha^5 \beta-855 \alpha^4 \beta^2-339 \alpha^3 \beta^3+1260 \alpha^4 x-1152 \alpha^3 \beta x+108 \alpha^2 x^2+1845 \alpha^4 y+2843 \alpha^3 \beta y\\
&\;+1464 \alpha^2 x y-2205 \alpha^2 y^2-5 \alpha \beta y^2+(59/2) y^3-(4063/3) y^2 z-12306 y z^2+18894 z^3\\
&\;+792 \alpha^7-1386 \alpha^6 \beta-1026 \alpha^5 \beta^2-531 \alpha^4 \beta^3+1512 \alpha^5 x-2268 \alpha^4 \beta x-9 \alpha^3 x^2+3204 \alpha^5 y+5112 \alpha^4 \beta y\\
&\;+4611 \alpha^3 x y-6407 \alpha^3 y^2-87 \alpha^2 \beta y^2+201 \alpha y^3-17941 \alpha y z^2+27317 \alpha z^3-(952/3) \beta z^3\\
&\;+495 \alpha^8-1386 \alpha^7 \beta-756 \alpha^6 \beta^2-612 \alpha^5 \beta^3+1260 \alpha^6 x-3024 \alpha^5 \beta x-459 \alpha^4 x^2+3990 \alpha^6 y+6282 \alpha^5 \beta y\\
&\;+9231 \alpha^4 x y-12342 \alpha^4 y^2-(775/2) \alpha^3 \beta y^2+615 \alpha^2 y^3+(39845/6) y^2 z^2-27565 y z^3+26133 z^4\\
&\;+220 \alpha^9-990 \alpha^8 \beta-234 \alpha^7 \beta^2-588 \alpha^6 \beta^3+720 \alpha^7 x-2772 \alpha^6 \beta x-1080 \alpha^5 x^2+3600 \alpha^7 y\\
&\;+5292 \alpha^6 \beta y+12648 \alpha^5 x y-16656 \alpha^5 y^2-(1757/2) \alpha^4 \beta y^2+(2403/2) \alpha^3 y^3\\
&\;-(93083/2) \alpha y z^3-(118563/2) \alpha z^4+(52961/2) \beta z^4\\
&\;+66 \alpha^{10}-495 \alpha^9 \beta+135 \alpha^8 \beta^2-522 \alpha^7 \beta^3+270 \alpha^8 x-1728 \alpha^7 \beta x-1512 \alpha^6 x^2+2340 \alpha^8 y+2934 \alpha^7 \beta y\\
&\;+12516 \alpha^6 x y-16170 \alpha^6 y^2-1104 \alpha^5 \beta y^2+(3783/2) \alpha^4 y^3+9131 y^2 z^3-15564 y z^4+14298 z^5\\
&\;+12 \alpha^{11}-165 \alpha^{10} \beta+195 \alpha^9 \beta^2-414 \alpha^8 \beta^3+60 \alpha^9 x-702 \alpha^8 \beta x-1566 \alpha^7 x^2+1070 \alpha^9 y+918 \alpha^8 \beta y\\
&\;+9426 \alpha^7 x y-11346 \alpha^7 y^2-644 \alpha^6 \beta y^2+2628 \alpha^5 y^3-16382 \alpha y z^4-14284 \alpha z^5+(24554/3) \beta z^5\\
&\;-36 \alpha^{11} \beta+99 \alpha^{10} \beta^2-255 \alpha^9 \beta^3-180 \alpha^9 \beta x-1296 \alpha^8 x^2+330 \alpha^{10} y+35 \alpha^9 \beta y+5724 \alpha^8 x y\\
&\;-5688 \alpha^8 y^2+99 \alpha^7 \beta y^2+2814 \alpha^6 y^3-(3587/3) y^2 z^4-6357 y z^5+6466 z^6\\
&\;+36 \alpha^{11} \beta^2-99 \alpha^{10} \beta^3-765 \alpha^9 x^2+60 \alpha^{11} y-99 \alpha^{10} \beta y+2775 \alpha^9 x y-1745 \alpha^9 y^2+576 \alpha^8 \beta y^2\\
&\;+2301 \alpha^7 y^3+6287 \alpha y z^5+28841 \alpha z^6-(26296/3) \beta z^6\\
&\;-36 \alpha^{11} \beta^3-297 \alpha^{10} x^2-48 \alpha^{11} \beta y+1353 \alpha^{10} y^2+(2995/2) \alpha^9 \beta y^2+(13617/2) \alpha^8 y^3\\
&\;-(10889/2) y^2 z^5-1698 y z^6+1333 z^7\\
&\;-108 \alpha^{11} x^2+696 \alpha^{11} y^2-(8915/2) \alpha^9 y^3+4078 \alpha y z^6+8370 \alpha z^7-3253 \beta z^7\\
&\;+202 y z^7-619 z^8+
(64/3) \beta z^8.
\end{split}
\]
\end{proposition}

\begin{remark}\label{rem:pointclass}
From the Poincar\'e polynomial (Proposition \ref{prop:ppolyM}), $\chi_{\mathrm{top}}(\rM) = 192$. By Gauss-Bonnet theorem, 
\[
	\int_{\rM} c_{17}(\cT_{\rM}) = \chi_{\mathrm{top}}(\rM) = 192.
\]
Therefore the point class in $\rA^{17}(\rM)$ is $\beta z^{8}/9$. 
\end{remark}

\subsection{Effective cycles}

In this section, we compute some effective cycle classes on $\rM$. 

The first natural effective cycle is the Brill-Noether locus. The blow-up center $\cC_{4} = \{F\in \rM|\;\dim \rH^{0}(F) = 2\}$ is a codimension two Brill-Noether locus. 

\begin{proposition}\label{prop:BNlocus}
For $r : \rM^{+} \to \rM$, 
\[
	r^{*}[\cC_{4}] = \tau(\tau + \rho + \frac{1}{3}b_{1}) = z.
\]
\end{proposition}

\begin{proof}
Consider the blow-up diagram \eqref{eqn:blowupdiagramforM}:
\[
	\xymatrix{E \ar[r]^{i} \ar[d]_{r'} & \rM^{+} \ar[d]^{r}\\ 
	\cC_{4} \ar[r]^{j} & \rM.}
\]
By \cite[Proposition 6.7.(a)]{Ful98}, $r^{*}[\cC_{4}] = i_{*}(c_{1}(t^{*}\cN) + c_{1}(\cO_{\PP \cN}(1))) $ where $\cN := \cN_{\cC_{4}/\rM}$. Then 
\[
\begin{split}
	r^{*}[\cC_{4}] 
	&= i_{*}(c_{1}({r'}^{*}(\cN))+c_{1}(\cO_{\PP \cN}(1)))
	= i_{*}({r'}^{*}(k-2\eta)+i^{*}c_{1}(\cO(-E)))\\
	&= i_{*}({r'}^{*}j^{*}((-\frac{1}{3}b_{1}+\rho)-2(\rho+\tau))
	-i^{*}c_{1}(\cO(E)))\\	
	&= i^{*}(i^{*}r^{*}((-\frac{1}{3}b_{1}+\rho)-2(\rho+\tau)))
	-i_{*}i^{*}c_{1}(\cO(E))\\
	&=-\tau((-\frac{1}{3}b_{1}+\rho)-2(\rho+\tau))-(-\tau)^{2}
	=\tau(\tau+\rho+\frac{1}{3}b_{1}).
\end{split}
\]
\end{proof}

The class $\alpha$ can be interpreted in the following way. 
\begin{proposition}\label{prop:alphaclass}
Fix a point $p$ on $\PP^{2}$. Let $S$ be the locus of $F\in \rM$ such that $p \in \mathrm{supp}(F)$. Then,
\[
	S = \alpha.
\]
\end{proposition}
\begin{proof}
By Lemma \ref{lem:twoChernclasses} and \cite[Lemma 3.1]{Woo13}, for the Fitting map $\pi : \rM \to |\cO_{\PP^{2}}(4)|$, 
\[
	12\alpha = -K_{\rM} = 12\pi^{*}(\cO_{|\cO_{\PP^{2}}(4)|}(1))
	= 12S. 
\]
\end{proof}

The next two effective cycles are obtained from the cokernel of the structure map $\cO_{\PP^{2}} \stackrel{s}{\to} F$. Note that for $F \in \rM \setminus (\cC_{4}\cup r(W))$ (see Definition \ref{def:EandW}), $\dim\rH^{0}(F) = 1$ and the cokernel $Q_{F}$ of $\cO \stackrel{s}{\to} F$ has a finite support.
\begin{definition}
\begin{enumerate}
\item Let $O$ be the closure of the locus of $F \in \rM$ such that $Q_{F}$ contains a fixed point. 
\item Let $L$ be the closure of the locus of $F \in \rM$ such that $Q_{F}$ meets a fixed line. 
\end{enumerate}
\end{definition}

In the remaining of this section, we will compute their numerical classes. 

We may construct some natural effective cycles on $\rN$ from Schubert calculus. For an embedding $\iota : \rN \hookrightarrow \mathrm{Gr}(3,\mathrm{Sym}^{2}V^{*})$ (Remark \ref{rem:setmapn}), $\cE \cong \iota^{*}\cS$ where $\cS$ is the tautological subbundle (\cite[Section 3.2]{Tjo98}). On the other hand, the Chern classes of $\cS$ and that of the universal quotient bundle $\cQ$ can be described in terms of Schubert cycles as the following:
\[
	c_{i}(\cS) = (-1)^{i}\sigma_{1,1,\cdots, 1}, 
	\quad c_{i}(\cQ) = \sigma_{i}.
\]
Thus $b_{1} = -\sigma_{1}$, $b_{2} = \sigma_{1,1}$, and $b_{3} = -\sigma_{1,1,1}$. Also we have
\[
	\sigma_{1,1}+\sigma_{2} = \sigma_{1}^{2}, \quad
	\sigma_{1,1,1} + \sigma_{1}(\sigma_{1,1}-\sigma_{2}) + \sigma_{3} = 0
\]
from $c(\cS)c(\cQ) = 1$.

We may describe Chern classes of $\cQ$ as the following way. For a fixed 3-dimensional subspace $W$ of $\mathrm{Sym}^{2}V^{*}$, $\sigma_{1}(W)$ is the locus of subspaces with nontrivial intersection with $W$. Let $W$ be a net of quadrics which has 3 non-linear base points $p_{1}, p_{2}, p_{3}$. Then for any $W' \in \sigma_{1}(W) \cap (\rN - \PP V^{*})$, there is a conic $C \in W' \cap W$. Then the net $W'$ defines three points on $C$. Thus $W'$ parameterizes three points $q_{1}, q_{2}, q_{3}$ such that all of $p_{i}, q_{j}$ lie on a conic. Evidently this is a codimension 1 condition since for every set of 5 points there is a unique conic passing through them.

Similarly, for a two dimensional subspace $T \subset \mathrm{Sym}^{2}V^{*}$, $\sigma_{2}(T)$ is the locus of $W'$ such that $W' \cap T$ is nontrivial. $T$ defines a pencil of conics, which provides 4 points $p_{1}, \cdots, p_{4}$. If $W' \in \sigma_{2}(T)$, then a conic $C \in W'$ contains all of $p_{i}$'s. Thus $\sigma_{2}(T) \cap (\rN - \PP V^{*})$ is the locus of three points $q_{1}, q_{2}, q_{3}$ such that there is a conic containing all of $p_{i}$ and $q_{j}$.

Finally, $\sigma_{3}$ can be regarded as the locus of 3 points lying on a fixed conic. 

Recall that there is a contraction map $t : \rH(3) \to  \rN$ (Proposition \ref{prop:hilbcon}). In \cite{ELB89}, Elencwajg and Le Barz provided explicit generators and the intersection pairing of $\rA^{*}(\rH(3))$. 

\begin{definition}[\protect{\cite[I, II]{ELB89}}]\label{def:cyclesonH}
\begin{enumerate}
\item The class $A$ is a divisor of $\rH(3)$ parametrizing subschemes which lies on a line. Note that $A$ is the exceptional divisor of $t$. 
\item The divisor $H$ is the locus of subschemes whose supports meet a fixed line. 
\item The codimension two class $h$ is the locus of subschemes whose supports contains a fixed point. 
\item The codimension two class $a$ is the locus of subschemes aligned with a fixed point. 
\item The codimension two class $p$ is the locus of subschemes that two points lies on a fixed line. 
\end{enumerate}
\end{definition}

The canonical divisor $K_{\rH(3)}$ is $-3H$. Since $K_{\rN} = 3b_{1}$, $t_{*}(H) = -b_{1}$, $t^{*}(-b_{1}) = H+A$ by adjunction. Therefore $-b_{1}$ is precisely the closure of the locus of three points whose supports meet a fixed line. Thus $-\beta = -b_{1}-3\tau \in \rA^{*}(\rM)$ is the closure of the locus parametrizing $F \in \rM$ such that $Q_{F}$ meets a fixed line, which is precisely $L$. 

Now the computation of intersection numbers of $t^{*}(\sigma_{2})$ with a basis of $\rA^{4}(\rH(3))$ is straightforward. For instance, $H^{2}h \cdot t^{*}(\sigma_{2}) = t_{*}(H^{2}h)\cdot \sigma_{2}$ is the number of triples $\{p_{1}, p_{2}, p_{3}\}$ of points on $\PP^{2}$ such that $p_{1}$ is fixed, $p_{2}$ (resp. $p_{3}$) lies on a line $\ell_{1}$ (resp. $\ell_{2}$), and $\{p_{i}\}$ and general four points $q_{1}, \cdots, q_{4}$ lies on a conic. Since there is only one conic $C$ that passes through all $q_{j}$'s and $p_{1}$, two remaining points $p_{2}$ and $p_{3}$ must be on the intersection of $(\ell_{1} \cup \ell_{2}) \cap C$. Therefore $H^{2}h \cdot t^{*}(\sigma_{2}) = 4$. From similar geometric argument, we can obtain 
\[
	h^{2}\cdot t^{*}(\sigma_{2}) = ha \cdot t^{*}(\sigma_{2}) = 
	H^{2}\cdot t^{*}(\sigma_{2}) = 0, hp\cdot t^{*}(\sigma_{2}) = 1.
\]
We leave the computation to the reader. Then from the intersection numbers in \cite[p.112]{ELB89}, we obtain 
\[
	t^{*}(\sigma_{2}) = H^{2} + HA - 2h + a - p.
\]
From $\sigma_{2} = b_{1}^{2} - b_{2}$, $t^{*}(-b_{1}) = H + A$, and $A^{2} + HA = 3a$ (\cite[p.112]{ELB89}), $t^{*}(b_{2}) = 2a + 2h + p$. 

On the other hand, from the blow-up formula of Chow ring (\cite[Theorem 1 in Appendix]{Kee92}), 
\[
	\rA^{*}(\rH(3)) = \rA^{*}(\rN)[A]/\langle A(b_{1}^{2}-3d_{2}), 
	A(b_{2}-d_{2}), A^{4}+c_{1}(\cN)A^{3}+c_{2}(\cN)A^{2}+
	c_{3}(\cN)A+[\PP V^{*}]\rangle.
\]
So the dimension of the kernel of $\cdot A : \rA_{4}(\rH(3)) \to \rA_{3}(\rH(3))$ is two and it is generated by $t^{*}(b_{1}^{2}-3d_{2})$ and $t^{*}(b_{2}-d_{2})$. On the other hand, from the relations in \cite[p.112]{ELB89}, the kernel is generated by $a+h$ and $H^{2}+HA - 3h - 3p$. Therefore we are able to obtain $t^{*}(d_{2}) = a+h+p$. In particular, $h = t^{*}(b_{2}-d_{2})$ and $b_{2}-d_{2} = t_{*}(h)$. Therefore on $\rM$, $x-y = b_{2}-d_{2}$ is the closure of the locus of $F$ such that $Q_{F}$ contains a fixed point, which is $O$. 

In summary, we have the following additional relations. 
\begin{proposition}\label{prop:relation}
\[
	L = -\beta, \quad O = x-y. 
\]
\end{proposition}

\begin{remark}\label{rem:effnef}
The effective cone $\mathrm{Eff}(\rM)$ is generated by $\alpha$ and $-\beta$ (\cite[Corollary 2.3]{CC15}). The nef cone $\mathrm{Nef}(\rM)$ is generated by $\alpha$ and $3\alpha-\beta$ (\cite[Theorem 1.2]{CC15}). 
\end{remark}


\section{Euler characteristic of line bundles on $\rM$}\label{sec:Eulerchar}

For $c \in \mathrm{K}(\PP^{2})$, let $\rM(c)$ be the moduli space of semistable sheaves of topological type $c$. For notational simplicity, suppose that the stability coincides with semistability, thus there is a universal family $\cF$ over $\rM(c)$. On $\mathrm{K}(\PP^{2})$, there is a quadratic form $\langle v, w\rangle := \chi(v\cdot w)$. Also there is a homomorphism
\begin{equation}\label{eqn:FMtransform}
	\lambda:c^{\perp}\subset \mathrm{K}(\PP^2)\lr 
	\mathrm{Pic}(\rM(c)), \quad 
	w\mapsto \mathrm{det}(\pi_{1 !}(\cF\otimes \pi_{2}^*(w)))
\end{equation}
where $\pi_{1} :\bM(c)\times \PP^2\lr \rM(c)$ and $\pi_{2}:\rM(c)\times \PP^2\lr \PP^2$ are the projection maps. 

The \emph{strange duality} of theta divisors for the projective plane is formulated as the following, in terms of \emph{generalized K-theoretic Donaldson numbers} (\cite[Section 1.2]{GNY09}). For $E \in \mathrm{Pic}(\rM)$, the genearlized K-theoretic Donaldson number of $E$ is defined by 
\[
	\chi(\rM, E) = \int_{\rM}ch(E)td(\rM).
\]

Let $c, v \in \mathrm{K}(\PP^{2})$ such that $\chi(c\cdot v) = 0$. Let $\rM(c)$ (resp. $\rM(v)$) be the moduli space of semistable sheaves with topological type $c$ (resp. $v$). The (strong) strange duality conjecture is that there is a strange duality isomorphism 
\[
	SD : \rH^{0}(\rM(c), \lambda(-v))^{*} 
	\stackrel{\cong}{\longrightarrow} \rH^{0}(\rM(v), \lambda(-c)).
\]
A weak strange duality conjecture is that the dimension of two relevant K-theoretic Donaldson numbers $\chi(\rM(c), \lambda(-v))$, $\chi(\rM(v), \lambda(-c))$ are same. 

By using the Chern class formula in Proposition \ref{prop:Chernclass}, in principle, we are able to compute all K-theoretic Donaldson numbers for $\rM$. This provides numerical data for strange duality conjectures. 

For $F \in \rM$, $c := ch(F) = (r(F), c_{1}(F), ch_{2}(F)) = (0, 4, -5)$. In this case, $\lambda$ in \eqref{eqn:FMtransform} is an isomorphism (\cite[Theorem 3.5]{LP93a}). The domain $c^{\perp}$ is $\{(-4k, k, -k/2+m)\;|\;k, m \in \ZZ\}$. Moreover, it is straightforward to check that $[\cO_{x}] = (0, 0, 1)$ and $-4[\cO_{\PP^{2}}]+[h] = (-4, 1, -1/2)$ for a point $x \in \PP^{2}$ and a line $h \subset \PP^{2}$, are two generators of $c^{\perp}$. Furthermore, $\lambda([\cO_{x}]) = \alpha$, $\lambda(-4[\cO_{\PP^{2}}]+[h]) = -3\alpha-\beta$ (\cite[Proposition 2.5]{CC15}). Therefore $\lambda((-4k, k, -k/2+m)) = (-3k+m)\alpha - k \beta$. 

By Bogomolov's inequality $2rc_{2}-(r-1)c_{1}^{2} \ge 0$ (\cite[Theorem 3.4.1]{HuLe10}) for semistable sheaves on a surface, to guarantee $\rM(v) \ne \emptyset$, $m \ge 5k/8$. Note that $\mathrm{Eff}(\rM)$ is generated by $\alpha, -\beta$ and $\mathrm{Nef}(M)$ is generated by $\alpha, 3\alpha-\beta$ (Remark \ref{rem:effnef}). $\alpha$ is not big because it is a pull-back of the hyperplane class in $|\cO_{\PP^{2}}(4)|$, but $3\alpha-\beta$ is big. By Kawamata-Viehweg vanishing theorem and $K_{\rM} = -12\alpha$, if $m \ge 6k-12$ and $k \ge 1$, then $\chi(E) = \rH^{0}(E)$. In particular, since $-K_{\rM}$ lies on an extremal ray of the nef cone, every big and nef divisor has vanishing higher cohomologies. 

Table \ref{tbl:Donaldsonnumber} gives some examples. The black numbers are equal to $\rH^{0}(E)$, due to Kawamata-Viehweg vanishing theorem. The grey numbers are Euler characteristics to whom we are unable to apply the vanishing theorem. 

\begin{table}[!ht]
\begin{tabular}{|c||r|r|r|r|r|}\hline
$m \setminus k$ & 1& 2 & 3 & 4 & 5\\ \hline \hline
1 & 0 & & & &\\ \hline
2 & 0 & 0 & \gray -2163 & &\\ \hline
3 & 20 & 0 &\gray -295 & \gray -3601488 &\\ \hline
4 & 315 & 0 &\gray  -21 & \gray -1326276 &\gray -463995675 \\ \hline
5 & 2643 & 0 &\gray 0 & \gray -430920 &\gray -238005495 \\ \hline
6 & 15681 & 148 & 0 & \gray -120840 &\gray -114238355 \\ \hline
7 & 73725 & 2310 & 0 & \gray -28224 &\gray -51151287\\ \hline
8 & 292020 & 19194 & 0 & \gray -5172 &\gray -21242133\\ \hline
9 & 1012180 & 112774 & 664 & \gray -664 &\gray -8110245 \\ \hline
10 & 3149964 & 525330 & 10059 & \gray -45 &\gray -2811753 \\ \hline
11 & 8965188 & 2063286 & 81375 & \gray 0 & \gray -869877\\ \hline
12 & 23659750 & 7098210 & 466731 & 2206 & \gray -234213 \\ \hline
13 & 58517910 & 21947166 & 2127240 & 32508 & \gray -52857 \\ \hline
14 & 136793730 & 62121942 & 8191071 & 256557 & \gray -9405\\ \hline
15 & 304306698 & 163198260 & 27675571 & 1439060 &\gray 4822 \\ \hline
16 & 647838516 & 402149484 & 84174255 & 6427680 &\gray 86316\\ \hline
17 & 1326076500 & 937334100 & 234701775 & 24298824&\gray 667842\\ \hline
18 & 2620208300 & 2080498556 & 608157605 & 80728712 &3676190\\ \hline
19 & 5014575000 & 4421930724 & 1479876465 & 241765632&16140585\\ \hline
20 & 9322330905 & 9041293260 & 3409833831 & 664574400&60064605\\ \hline
\end{tabular}
\medskip
\caption{K-theoretic Donaldson numbers of classes $(-3k+m)\alpha - k\beta$}\label{tbl:Donaldsonnumber}
\end{table}

\begin{remark}\label{rem:conjfor4}
If $k = 0$, we have a closed formula
\begin{equation}\label{eqn:Donaldsonpoly}
	\chi(\rM,m\alpha) = {m+11 \choose m}.
\end{equation}
Note that $\chi(\rM,m\alpha)$ is a polynomial of degree at most 17, so to prove \eqref{eqn:Donaldsonpoly}, it is sufficient to compare the values when $0 \le m \le 17$.
\end{remark}

\begin{conjecture}
Let $\alpha_{d} := \pi^{*}\cO_{|\cO_{\PP^{2}}(d)|}(1)$ for the Fitting map $\pi : \rM(d, 1) \to |\cO_{\PP^{2}}(d)|$. Then
\[
	\chi(\rM(d,1), m\alpha_{d}) = {m+3d-1 \choose m}
	= {m+\dim |\cO_{\PP^{2}}(d)|-g \choose m},
\]
where $g$ is the dimension of a general fiber of $\pi$. 
\end{conjecture}

This conjecture is true for $d \le 4$.


\bibliographystyle{alpha}
\newcommand{\etalchar}[1]{$^{#1}$}

\end{document}